\documentclass[a4paper, reqno,12pt]{amsart}

\usepackage{tikz, tikz-3dplot}
\usepackage{amsmath, amsthm, amssymb,graphics}
\usepackage{graphicx}
\usepackage{float}
\theoremstyle{plain}
\newtheorem{te}{Theorem}[section]
\newtheorem{lem}[te]{Lemma} 
\newtheorem*{lem*}{Lemma A}

\newtheorem{pr}[te]{Proposition}

\theoremstyle{remark}
\newtheorem*{re}{Remark}
\newtheorem*{ack*}{\textbf{Acknowledgement}}

\def\ni{{\noindent}}
\def\l{{\bf l}}\def\e{{\bf e}}

\def\tt{{\tilde{\tau}}}
\def\t{{\tau}}

\def \G{{\Gamma}}

\def\th{{\theta}}

\def\tp{{\tilde{\psi}}}
\def \tP{{\tilde{\Phi}}}
\def \tPs{{\tilde{\Psi}}}

\def \ta{{\tilde{\alpha}}}

\def \l{{\lambda}}

\def\R{{\mathbb R}}

\def\N{{\mathbf N}}
\def\Nb{{\mathbb N}}

\def\S{{\mathbb S}}

\def \bq{{\bar{q}}}

\def\M{{\mathbb M}}
\def\A{{\mathbf A}}

\def\L{{\mathbf L}}
\def\x{{\mathbf x}}
\def \e{{\mathbf e}}

\def\Dec{{\operatorname{Dec}}}

\def\nint{\mathop{\diagup\kern-13.0pt\int}}

\def \tQ{{\tilde{Q}}}
\def \Rc{{\mathcal R}} 
 
\def\Cc{{\mathcal C}}\def\Nc{{\mathcal N}}
\def\Tc{{\mathcal T}}
\def\Bc{{\mathcal B}}\def\Rc{{\mathcal R}}
\def\Ac{{\mathcal A}}
\def\Pc{{\mathcal P}}
\def\Sc{{\mathcal S}}\def\Vc{{\mathcal V}} 
\def \Ec{{\mathcal E}}

\def\emph#1{{\it #1}}

\begin{document}

		\author{D\'ominique Kemp}
		\address{Department of Mathematics, Indiana University,  Bloomington IN}
		\email{dekemp@iu.edu}

		\title{Decoupling via Scale-Based Approximations of Limited Efficacy}
		\maketitle

\begin{abstract}We consider the decoupling theory of a broad class of $C^5$ surfaces $\M \subset \R^3$ lacking planar points. In particular, our approach also applies to surfaces which are not graphed by mixed homogeneous polynomials. The study of $\M$ furnishes opportunity to recast iterative linear decoupling in a more general form. Here, Taylor-based analysis is combined with efforts to build a library of canonical surfaces (non-cylindrical in general) by which $\M$ may be approximated for decoupling purposes. The work presented may be generalized to the consideration of other surfaces not addressed.
\end{abstract}
\maketitle

\section{Background}
\label{s1}

Let $\M$ be a compact, $C^5$ surface in $\R^3$ and $\N: \M \rightarrow \S^2$ a unit normal vector field on $\M$. Gaussian curvature is defined for $\M$ as $$K = \det \text{d}\N.$$ Subsets $\th \subset \M$ shall be called caps, and we shall consider their $\delta$-neighborhoods generated by $\N$: $$\Nc_\delta(\th) = \{p + v\N(p): p \in \th, v \in [-\delta, \delta]\}.$$Throughout this paper, we shall be mentioning Fourier projections onto such sets, so let us define $P_S f$ as the Fourier projection of a function $f$ onto $S \subset \R^3$: $$P_S f(x) = \int_S \hat{f}(\xi) e(x \cdot \xi) d\xi,$$ where $e(r)$ represents $e^{2\pi i r}$. Lastly, within the context of this paper, $\th$ shall be referred to as rectangular if its projection onto a plane tangent to $\M$ at some point in $\th$ is a rectangle. 

In \cite{BD3} and \cite{BD4}, Bourgain and Demeter introduced $l^r(L^p)$ decoupling theory for hypersurfaces in $\R^n$. Concerning the case $n=3$, we have from their work the following theorem, which holds more generally for $C^3$ hypersurfaces in $\R^3$.

\begin{te} 
\label{t0.1}
Assume that $K$ vanishes nowhere on $\M$. Then, for each $\delta > 0$, there exists a partition $\Pc_\delta(\M)$ of $\M$ into maximally flat, square caps $\t$ of side length $\delta^{1/2}$ such that the following holds for all $2 \leq p \leq 4$. 

For each fixed $\epsilon > 0$, there exists $C_\epsilon > 0$ (depending only on $\epsilon$, the $C^3$ norm of $\M$, and a lower bound for $\l_1$ and $\l_2$) such that for all $\delta$

\begin{equation} \label{*} \|f\|_{L^p(\R^3)} \leq C_\epsilon \delta^{-\epsilon}|\Pc_\delta(\M)|^{1/2 - 1/p}(\sum_{\t \in \Pc_\delta(\M)} \|P_{\Nc_\delta(\t)} f\|_{L^p(\R^3)}^p)^{1/p} \end{equation} whenever $f$ is Fourier supported in $\Nc_\delta(\M)$. 

If $K > 0$ throughout $\M$, we have also for all $\epsilon > 0$ \begin{equation} \label{**}  \|f\|_{L^p(\R^3)} \leq C_\epsilon \delta^{-\epsilon} (\sum_{\t \in \Pc_\delta(\M)} \|P_{\Nc_\delta(\t)} f\|_{L^p(\R^3)}^2)^{1/2} \end{equation} for all $f$ Fourier supported in $\Nc_\delta(\M)$.

\end{te} In particular, when $\M$ is the paraboloid, the elements of $\Pc_\delta(\M)$ project down onto squares of side length $\delta^{1/2}$.

We call \eqref{*} and \eqref{**} respectively $\ell^p$ and $\ell^2$ decoupling at scale $\delta$ when they hold for partitions $\Pc_\delta(\M)$ comprised of flat caps. A flat cap $\t$ is characterized as a subset of $\M$ that lies within $\delta$ of a plane $\Pi$ tangent to $\M$. In this case, we say that $\t$ is \emph{$\delta$-approximated} by $\Pi$. 

When $\Pi$ $\delta$-approximates $\t$, $\tt = \Nc_\delta(\t)$ is in turn approximated by a rectangle $R$ contained within it. In other words, in addition to $R \subset \tt$, an $O(1)$-enlargement of $R$ contains $\tt$. Consequently, the following proposition shows that we cannot non-trivially partition $\t$ for the purpose of $\ell^2$ decoupling.

\begin{pr}
\label{p0.1}
Let $L$ be a line segment in $\R^n$ of length $\sim 1$. For each $0 \le \delta, N^{-1}<1$, let $\Pc_{\delta,N}$ be a partition of the $\delta$-neighborhood $\Nc_\delta(L)$ of $L$ into $\sim N$ cylinders $T$  with length $N^{-1}$ and radius $\delta$.

For $p>2$, let $D(\delta,N,p)$ be the smallest constant such that
\begin{equation}
\label{8.7}
\|f\|_{L^p(\R^n)}\le D(\delta,N,p)(\sum_{T\in\Pc_{\delta,N}}\|P_T f\|_{L^p(\R^n)}^2)^{1/2}
\end{equation}
holds for all $f$ Fourier supported on $\Nc_\delta(L)$. Then
$$D(\delta,N,p)\sim N^{\frac12-\frac{1}{p}},$$
and (approximate) equality in \eqref{8.7} can be achieved by using a smooth approximation of $1_{\Nc_\delta(L)}$.
\end{pr}

This statement shows why when obtaining $\ell^2$ decoupling at least, we should aim for maximally flat caps $\t$. For $\t$ to be maximally flat means that any enlargement of it exceeds the $\delta$-neighborhood of any plane. In other words, those larger caps are curved at scale $\delta$ and may be decoupled without violating the proposition. However, we cannot decouple maximally flat caps any further without incurring substantial loss. In this sense, decoupling partitions whose elements are each maximally flat may be loosely labeled \textit{maximal}, signaling that those elements cannot be partitioned further non-trivially. We should mention that there is a notion of small cap decoupling, which allows for smaller flat caps to be taken but only in the framework of $\ell^p$ decoupling. Demeter, Guth, and Wang address this for the parabola, two-dimensional cone, and moment curve in $\R^3$ in \cite{DGW}. 

Theorem \ref{t0.1} does not explicitly address decoupling for manifolds $\M$ having vanishing curvature. If $K \equiv 0$, decoupling for $\M$ is resolved in \cite{K} for the three-dimensional, non-planar case. When $K$ does not vanish everywhere, there is reason to believe that some form of decoupling still holds in all cases. Lojasiewicz's inequality and Puiseux series expansion (see \cite{BM} and \cite{KP}) show that curvature increases at least exponentially away from the vanishing set $\Vc$. Therefore, away from $\Vc$, the setting has some semblance to that of Theorem \ref{t0.1}. For example, concerning real-analytic surfaces, the $\ell^2$ decoupling theorem for real analytic curves (\cite{D}, Section 12.6) and also the decoupling theorem for real analytic surfaces of revolution \cite{BDK} is obtained using this insight. Furthermore, by similar methods combined with polynomial case analysis, Li and Yang \cite{LY} have recently proved $\ell^2$ and $\ell^p$ decoupling for the mixed homogeneous polynomial surfaces in $\R^3$ with bounds depending on the polynomial degree.

Our paper addresses another class of surfaces of vanishing curvature that has some overlap with the mixed homogeneous polynomial surfaces. Some of the surfaces covered by Theorem \ref{t0.2} are mixed homogeneous, but this theorem does not address all such surfaces as does \cite{LY}. Examples covered by Theorem \ref{t0.2} are all polynomial surfaces of the form \begin{equation} \label{poly} (\xi_1, \xi_2) \mapsto \sum_{i,j=0}^m A_{ij}\xi_1^i \xi_2^j \end{equation} with $A_{02}$ and $A_{21}$ both nonzero. Of course, some of these may be mixed homogeneous. The theorem extends, however, to many other surfaces too that merely appear locally of the form \eqref{poly}. These are the surfaces whose lines of curvature are all curved. The torus is one additional example of this class, but there are many perturbations of the torus that also fit this description. As well, more exotic examples may be constructed, such as by extracting from the class of canal surfaces \cite{GS}.

In the sequel, there will be some slight overlap with the methods pursued in \cite{LY}. In particular, we make use of the usual ingredients of dyadic decomposition, cylindrical decoupling, iteration, and rescaling. Whereas in \cite{LY} the argument proceeds by way of insightful polynomial manipulations (introducing dependence upon degree) and cylindrical decoupling, our paper pioneers the use of more general approximating surfaces rooted in the Taylor structure of a manifold. In particular, such surfaces generally are not cylindrical, so their decoupling theory may necessarily have an $\ell^p(L^p)$ component. Because of the nature of the $\ell^p$ decoupling constant, this detail could disrupt decoupling by way of iteration. Critically, the iterative decoupling that will be done here will always be in $\ell^2(L^p)$ rather, a nontrivial matter resolved via Taylor-based, geometric arguments. Additional introductory detail concerning this new form of iteration and its first demonstration are given respectively in Section \ref{s4.3} and in the proof of Proposition \ref{pr2.0}.

\section{The main result}

As has been portrayed in previous work, introducing a nonempty set $\Vc \subset \M$ where the Gaussian curvature vanishes alters the form of $\Pc_\delta(\M)$. Its elements are no longer all square-like but rather rectangular whose eccentricity increases as $\Vc$ is approached. Furthermore, the presence of $\Vc$ makes it possible for $\M$ to have regions $\M^+ \subset \M$ and $\M^- \subset \M$ of positive and negative Gaussian curvature respectively. Thus, a combined $\ell^2$ and $\ell^p$ decoupling inequality for $\M$ is anticipated.

When entertaining scenarios where $\Vc \ne \emptyset$, it is natural to wonder what might hold true if curvature yet persists within certain lower-dimensional submanifolds of $\M$. After all, cylindrical decoupling, defined in Section \ref{s3}, has been used in contexts of vanishing Gaussian curvature where curved arcs foliated the region to be decoupled (see \cite{BDK} and \cite{K} for example). In the general case, we look for a canonical choice of curves having nonzero curvature everywhere and foliating $\M$ in some sense. In \cite{BD3}, there is a hint toward lines of curvature (see Section \ref{s2}), when the authors obtain a nice model, the elliptic paraboloid, for a general hypersurface $S$. Thus, even when $K$ vanishes at points $q \in \M$, we are led to expect decoupling, so long as the corresponding line of curvature nevertheless has nonzero curvature at each $q$.

In this paper, it is our main goal to determine the $\ell^2$ and $\ell^p$ decoupling theory for any $C^5$ manifold $\M$ satisfying such a condition. We will only concern ourselves with partitioning over maximally flat caps; the subject of small cap decoupling for $\M$ is left open and appears to be an interesting problem. 

In keeping with convention, we initially partition $\M$ into regions $\Ac_k$ where the Gaussian curvature $K$ is approximately $2^{-k}$ in magnitude ($2^{-k} \geq \delta^{1/2}$) and treat each $\Ac_k$ individually. Such a decomposition is advisable furthermore since it renders a sharper $\ell^p$ decoupling over the region where $K < 0$. 

\begin{te} 
\label{t0.2}
Let $\M$ be a compact $C^5$ surface in $\R^3$ having no planar points and satisfying that its lines of curvature, as curves, have nonzero curvature at every point. For each $\delta > 0$, there exists a partition $\Pc_\delta(\M)$ of $\M$ into boundedly overlapping, flat rectangular caps $\t \subset \M^+ \cup \{|K| \leq \delta^{1/2}\}$ and $\t_k \subset \M^- \cap \Ac_k$ having dimensions at most $\delta^{1/4} \times \delta^{1/2}$ such that the following holds for all $2 \leq p \leq 4$. 

For every fixed $\epsilon > 0$, there exists $C_\epsilon > 0$ (depending only on $\epsilon$ and the $C^5$ norm of $\M$) such that uniformly in $\delta$

 \begin{equation} \label{***} \|f\|_{L^p(\R^3)} \leq C_\epsilon \delta^{-\epsilon}((\sum_{\t} \|P_{\Nc_\delta(\t)} f\|_{L^p(\R^3)}^2)^{1/2} + \sum_k |\{\t_k\}|^{1/2-1/p}(\sum_{\t_k} \|P_{\Nc_\delta(\t_k)}f\|_{L^p(\R^3)}^p)^{1/p}) \end{equation} for every $f$ that is Fourier supported in $\Nc_\delta(\M)$.

\end{te} Some examples covered by the theorem are the torus, first dealt with in \cite{BDK}, and various perturbations of it. Notably, the torus has regions of both positive and negative Gaussian curvature. In the appendix, we show that all $\M$ are structured in this way. Additional explicit examples for $\M$ are the graphs of polynomial surfaces, restricted to a neighborhood of the origin, defined by $$(\xi_1, \xi_2) \mapsto \sum_{i,j=0}^m A_{ij}\xi_1^i \xi_2^j$$ with $A_{02}$ and $A_{21}$ both nonzero. Proposition \ref{pr2.1} explicitly establishes the decoupling theory of the graph of $$\Psi(\xi_1, \xi_2) = A_{20} \xi_1^2 + \xi_2^2 + A_{30} \xi_1^3 + \xi_1^2 \xi_2 + A_{40} \xi_1^4, \qquad \;\;\;\;\;\; |\xi_i| \leq 1.$$

In the next section, we define the lines of curvature and characterize the manifolds $\M$ satisfying the hypothesis of Theorem \ref{t0.2}.

\begin{ack*} The author is sincerely grateful to his advisor Ciprian Demeter for his iterated challenge to find the clearest expression of mathematical thought. He is also appreciative of Ciprian's continued encouragements and motivation during the writing of this manuscript and prior.

\end{ack*}

\section{Geometric preliminaries}
\label{s2}

\subsection{Graph parametrization} In order to obtain Theorem \ref{t0.2}, we must first lay some groundwork concerning the structure of $\M$. We know that $\M$ locally is given as a graph over a plane $T_q(\M)$ tangent to $\M$ at a point $q \in \M$. The point $q$ is identified as $(0,0,0)$ with respect to the coordinate system introduced by $T_q(\M)$ and $\N(q)$. Furthermore, we take $Q$ to be a closed set containing the origin in its interior, the set $\tQ = \{(\xi_1, \xi_2, \psi(\xi_1, \xi_2)): (\xi_1, \xi_2) \in Q\}$ to be a closed neighborhood of $q$ in $\M$, and also $\psi$ to satisfy \begin{equation} \label{1.0} \psi(0,0) = 0, \qquad \nabla \psi(0,0) = 0.\end{equation} The reason for taking closed sets is to ensure that all continuous functions defined on $Q$ (or $\tQ$) are bounded, a point that will be significant later. Also, let us clarify that $\psi$ is just the locally defined function that gives height of points in $\M$ above the tangent plane $T_q(\M)$. We write $\Gamma_\psi$ for the graph of $\psi$, and in the sequel, $\tilde{B}$ will always be shorthand for $\Gamma_{\psi|_B}$ for subsets $B$ contained in the domain. Taking $\tQ$ smaller if necessary, we may ensure that $\tQ$ is the graph of any height function taken with respect to any tangent plane $T_{q'}(\M)$ for $q' \in \tQ$.

When we conceive of $\M$ locally as a graph $\tQ$, our definition of the $\delta$-neighborhood changes to $$\Nc_\delta(\tQ) = \{(\xi_1, \xi_2, \psi(\xi_1, \xi_2) + v): (\xi_1, \xi_2) \in Q, |v| \in [0, \delta)\}.$$ Indeed, $\Nc_{2\delta}(\G_{\psi|_{2Q}})$ covers the $\delta$-neighborhood as previously defined for $\delta$ sufficiently small. So we suffer no loss when we work with $\psi(2Q)$ still satisfying the above conditions. Let us also mention here that the graph parametrization allows for us to conceive of $K$ as essentially the Hessian of $\psi$. This is because of the formula \begin{equation} \label{*1} K = \frac{\psi_{11}\psi_{22} - \psi_{12}^2}{(1 + \psi_1^2 + \psi_2^2)^2},\end{equation} provided to us by differential geometry. As well, there is a standard formula for the sum of $\l_1$ and $\l_2$, called the \emph{mean curvature}, and it depends only on the first and second-order derivatives of $\psi$.

\subsection{Curvature parametrization} As well, we shall want $\tQ$ to be the image of a parametrization $\x = \x(u,v)$, defined on a closed set containing $(0,0)$ in its interior, that satisfies the following property. Namely, $\x(0,0) = q$, and \begin{equation} \label{1.1} \N'(u) = -\l_1 \x_u(u,v) \end{equation} \begin{equation} \label{1.2} \N'(v) = -\l_2 \x_v(u,v), \end{equation} where $\N(u)$ and $\N(v)$ are the restrictions of $\N$ to the coordinate curves $u \mapsto \x(u,v)$ and $v \mapsto \x(u,v)$ respectively. The reader may consult \cite{DoC} as a reference, in particular to confirm that such an $\x$ exists at any non-planar point. Note that $\l_1$ and $\l_2$ are the functions giving the eigenvalues of $d\N$, which is indeed a self-adjoint map. In the literature of differential geometry, $\l_1$ and $\l_2$ are called the \emph{principal curvatures} of $\M$, the coordinate curves of $\x$ are called \emph{lines of curvature}, and their tangent lines are the \emph{principal directions}. We shall call $\x$ the \emph{curvature parametrization} of $U$, and we say that $\l_1$ and $\l_2$ respectively \emph{correspond} to the curves $\x(\cdot, v)$ and $\x(u, \cdot)$. 

\subsection{Definition of principal expansions} It is straightforward to see that the Hessian of $\psi$ is equal to $d\N$ at $q$. Thus, its eigenvalues there are the principal curvatures at $q$. So, when we take the $\xi_1, \xi_2$ coordinates with respect to the principal directions at $q$, the third-order Taylor expansion of $\psi$ appears as \begin{equation} \label{1.3} \psi(\xi_1, \xi_2) = A_{2,0} \xi_1^2 + A_{0,2} \xi_2^2 + A_{3,0} \xi_1^3 + A_{2,1} \xi_1^2\xi_2 + A_{1,2} \xi_1 \xi_2^2 + A_{0,3} \xi_2^3 + O(|(\xi_1, \xi_2)|^4), \end{equation} where $A_{2,0} = \l_1(q), A_{0,2} = \l_2(q)$, and $A_{3,0}, A_{2,1}, A_{1,2}, A_{0,3} \in \R$. We call \eqref{1.3} the \emph{principal expansion} of $\psi$ at $q$.

Throughout this paper, we are only concerned with the decoupling theory for neighborhoods $\tQ$ having $K$ vanishing at a point $\bq \in \tQ$. Thus, we specialize $\tQ$ to such a neighborhood. Since $\bq$ is not planar, we may then state that $$\l_1(\bq) = 0,$$ while $$\l_2(\bq) \ne 0.$$ Restricting the size of $\tQ$ if necessary, we may arrange for $$\l_2 \sim 1$$ throughout $\tQ$. The Gaussian curvature is then $ \sim \l_1$ within $\tQ$.

It is helpful to look for connections between the two parametrizations for $\M$ that we have, namely the graph parametrization and the curvature parametrization. In particular, the hypothesis of Theorem \ref{t0.2} carries an implication for the expression \eqref{1.3}. From Lemma A in Section \ref{a1} of the Appendix, we know that the Taylor coefficient $A_{2,1}$ for a principal expansion \eqref{1.3} based at $\bq$ is nonzero,  since it also holds that $\l_1$ vanishes at $\bq$. Taking $U$ again smaller, we consequently have \begin{equation} \label{1.4} A_{2,1} \sim 1 \end{equation} for each principal expansion \eqref{1.3} taken at a point $q \in \tQ$.

\subsection{A canonical approximating surface for $\M$} We thus are beginning to get a view of the local geometry of $\M$ near the vanishing set $\Vc$ of $K$. So far, we know that it is characterized as a finite union of graphs whose functions satisfy $\eqref{1.4}$. In light of \eqref{1.4}, the $A_{2,1} \xi_1^2 \xi_2$ term introduces principal curvature along the $\xi_1$ direction, replacing the possibly very small curvature $A_{2,0}$. We would like to leverage this for decoupling purposes. An ideal way to do this is to construct from \eqref{1.3} a well approximating graph $G$ for $\tQ$ whose decoupling theory is more tangible and also facilitates induction on scales. Simply put, we want $G$ to decouple some $\delta^{-C}$ enlargement of $\Nc_\delta(\M)$ into curved boxes within which $G$ $\delta$-approximates $\M$, that is, lies within $\delta$ of $\M$.  Of course, the function defining $G$ should include the $A_{0,2} \xi_2^2$ term in addition to $A_{2,1} \xi_1^2 \xi_2$, so that the $\xi_2$ length may also be partitioned. If we were to stop here, the decoupling theory of \begin{equation} \label{1.5} (\xi_1, \xi_2) \mapsto A_{0,2} \xi_2^2 + A_{2,1} \xi_1^2 \xi_2 \end{equation} gives caps of dimensions at most $\delta^{1/4} \times \delta^{1/2}$ (see Proposition \ref{pr2.1}). But then, we are faced with the obstacle that there is no enlargement of $\Nc_\delta(\M)$ that \eqref{1.5} decouples into caps within which \eqref{1.5} is $\delta$-approximating. We are forced to include $A_{2,0} \xi_1^2, A_{3,0} \xi_1^3$ and even a fourth-order term $A_{4,0} \xi_1^4$. Within caps of dimensions $\delta^{1/5} \times \delta^{2/5}$, we know that the surface graphed by \begin{equation} \label{1.6} \Psi(\xi_1, \xi_2) = A_{2,0} \xi_1^2 + A_{0,2} \xi_2^2 + A_{3,0} \xi_1^3 + A_{2,1} \xi_1^2\xi_2 + A_{4,0}\xi_1^4 \end{equation} is $O(\delta)$-approximating for $\M$, where the constant is the $C^5$ norm of $\M$. Therefore, it suffices for our purposes to derive the decoupling theory of this graph. We obtain its $\ell^2-\ell^p$ decoupling theory in Section \ref{s6}.

\section{Tools for Acquiring $\Pc_\delta(\M)$} \label{s3}

Let us clarify why it suffices to find $\Pc_\delta(\tQ)$ for $\tQ$ as described in the previous section. Since $\Vc$ is compact, we may cover it by finitely many such neighborhoods $\tQ_i$, where each $Q_i$ is taken to be  a closed square. Once we uncover the decoupling partition $\Pc_\delta(\tQ_i)$ of $\Nc_\delta(\tQ_i)$, Theorem \ref{t0.2} then follows since Theorem \ref{t0.1} directly applies to $\M \backslash (\bigcup_i \tQ_i)$, and the various elements of $\bigcup \Pc_\delta(\tQ_i) \cup \Pc_\delta(\M\backslash (\bigcup_i \tQ_i))$ may then be ``patched" together into a uniform decoupling partition by standard Fourier projection theory.

\subsection{Basic decoupling} What precedes the patching process is an obvious predecessor of decoupling that we shall refer to as \emph{basic decoupling}. If we ever want to decouple a cap $S \subset \M$ into $O(1)$ many caps $S_1, \dots, S_N$, we may apply first the triangle inequality and then H\"{o}lder's inequality to obtain $$\|P_{\Nc_\delta(S)}f\|_p \leq N^{1 - 1/r}(\sum_{i=1}^N \|P_{\Nc_\delta(S_i)} f\|_p^r)^{1/r}.$$ In this way, we may initially write $$\|f\|_p \lesssim_\M ( \|P_{\Nc_\delta(\M \backslash (\bigcup_i \tQ_i))} f\|_p^r + \sum_i \|P_{\Nc_\delta(Q_i)} f\|_p^r)^{1/r},$$ $r = 2$ or $p$. Basic decoupling will recur repeatedly in the sequel as a means of decoupling possible $O(1)$ enlargements of the caps that will be desired. 

\subsection{Cylindrical decoupling} There is another tool for obtaining decoupling inequalities, which we shall need for the proof of Theorem \ref{t0.2}. It is also proven using elementary facts from analysis. The method is called \emph{cylindrical decoupling} and allows for us to extend decoupling results in $n$ dimensions to $n+1$ dimensions. From an inequality $$\|f\|_p \leq C (\sum_i \|P_{S_i} f\|_p^r)^{1/r}$$ that holds for all $f$ Fourier supported in $\bigcup S_i \subset \R^n$, we may deduce the inequality $$\|f\|_p \leq C (\sum_i \|P_{S_i \times \R} f\|_p^r)^{1/r}$$ for all $f$ Fourier supported in $\bigcup S_i \times \R$. The proof proceeds from Fubini's theorem and Minkowski's inequality.

\subsection{Type II iterative decoupling} \label{s4.3} One of the key contributions of this paper is an additional tool for decoupling that is more novel to the field. It is \emph{type II iterative decoupling}, so named in view of its predecessor which we call here \emph{type I iterative decoupling}. Type I iterative decoupling first appeared in \cite{PS} and was also used in Section 7 of \cite{BD3} to extend the $\ell^2$ decoupling of the paraboloid to arbitrary hypersurfaces $S$ with all positive principal curvatures. What occurs there is that $S$ is locally modeled by the surface defined by the second-order Taylor terms from \eqref{1.3}, and the remaining error is such as to allow for decoupling into smaller lengths via that model surface. In that scenario, the argument is carried out as an induction on scales; however, in Section 8 of \cite{BD3} and Section 3 of \cite{K}, we see that it can be performed as an iterative scheme where at each step we apply the decoupling theory of our model surface. 

The latter method correlates better with this current paper. In Section \ref{s6}, we will utilize the implied presence of model surfaces $\G_\Phi$ delivered to us by certain terms within \eqref{1.3}. The remaining Taylor terms will be treated as error that is in fact self-improving, allowing us to steadily decrease the size of our caps via decoupling with respect to $\G_\Phi$. Where type II decoupling crucially differs from type I decoupling lies in the limited efficacy of these scale-based surface approximations. On the one hand, they are limited concerning the scale to which they decouple $\|P_{\Nc_\delta(\M)} f\|_p$. At best, only scales at the rescaling threshold can be reached, and sometimes larger intermediate scales instead are attained. Furthermore, type II iterative decoupling never tangibly reaches the target dimensions that are in view. Rather, these are attained in the limit.

Though the iterations occur with respect to steadily varying surfaces, standard Fourier projection theory enables us to preserve the rectangular aspect of the caps $\th_i$. Note that decoupling with respect to some surface $\G = \G_\Phi$ partitions $\Nc_\delta(\th)$ by rectangular slabs $\Bc$ whose orientations depend on the principal directions of $\G$ (not $\M$). Since $\G$ varies throughout $\M$ (because \eqref{1.3} varies throughout $\M$), we will ultimately obtain a rather irregular, at best many-sided polygonal shape if we are not careful. However, supposing that the length dimensions of $\th_i$ at each step are sufficiently small, there will be $O(1)$ overlap between the $\Bc$ and rectangular caps $\th_{i+1} \subset \th_i$ parallel to the principal directions of $\M$. Thus, Fourier projection theory for rectangular regions enables us to transform the decoupling via $\G$ into one that partitions $\th_i$ into the $\th_{i+1} $. The $\th_{i+1}$ will have the same length dimensions as $\Bc$.

\subsection{Affine equivalence} Lastly, given complex-valued functions $f$ and $g$, we say that $f$ is \emph{affinely equivalent} to $g$ if there exists an invertible affine map $$\A(\xi) = \L \xi + a$$ such that $$\hat{g} =  \det(\L)\hat{f} \circ \A.$$ Change of variables on the frequency side shows that $$|P_S f(x)| = |P_{\A^{-1}(S)}g(\L^T x)|$$ for all $S \subset \R^3$. The latter is Fourier supported in $\A^{-1}(S)$, and so a decoupling inequality holding for $\A^{-1}(S)$ $$ \|P_{\A^{-1}(S)}g\|_p \leq C_1 (\sum_{\A^{-1}(S_i) \subset \A^{-1}(S)} \|P_{\A^{-1}(S_i)}g\|_p^r)^{1/r} $$ via change of variables on the spatial side implies the analogous one $$ \|P_{S}f\|_p \leq C_1(\sum_{S_i \subset S} \|P_{S_i}f\|_p^r)^{1/r}$$for $S$. This observation is valuable to decoupling theory. Our first example of an affine equivalence is provided in the proof of the rescaling lemma, which we now present.

\section{The role of rescaling} \label{s4}

In this section, we begin to prove decoupling results which culminate in Theorem \ref{t0.2}. The range for the Lebesgue index $p$, both here and in the following sections, will always be $2 \leq p \leq 4$. Although we begin with fairly small scales in the next lemma, we progressively increase the size of the caps to be decoupled over until ultimately the model surface $\G_\Psi$ is resolved.

The following lemma launches our investigation into the $\ell^2$ and $\ell^p$ decoupling theory of our model surface $\G_\Psi$. It applies to any surface lacking planar points, and so is written in corresponding generality. The proof is a formal extension of the rescaling arguments in \cite{BDK}. 

\begin{lem}{(Rescaling lemma)} \label{l2.1} Let $\tQ$ be the cap \eqref{1.3} projecting onto $[-B_1, B_1] \times [-B_2, B_2]$ and subject to $A_{0,2} \sim 1$ but with $A_{2,1}$ momentarily not assumed to be $\sim 1$ here. We assume that the Gaussian curvature $K$ of $\tQ$ is essentially constant throughout and also crucially that \begin{equation} \label{1.7} B_2 \leq |K|^{1/2} B_1. \end{equation}

 \begin{enumerate} \item Suppose that $|A_{3,0}| \lesssim \sqrt{|A_{2,0}|}$. Then, if $B_1 \leq \sqrt{|A_{2,0}|}$, (maximally) flat cap decoupling holds within $\tQ$.

\item Suppose instead that $\sqrt{|A_{2,0}|} \lesssim |A_{3,0}|$. Then, if $B_1 \leq |\frac{A_{2,0}}{A_{3,0}}|$, flat cap decoupling holds within $\tQ$. 
\end{enumerate} \bigskip
\ni The maximally flat caps obtained for the decoupling inequality, or rather their projections onto $T_p(\tQ)$,  have dimensions at most $|A_{2,0}|^{-1/2}\delta^{1/2} \times \delta^{1/2}$ relative to the principal directions at $p$. As well, the decoupling inequality is either $\ell^2(L^p)$ or $\ell^p(L^p)$ according as $K$ is positive or negative.

\end{lem}

\begin{proof} 

Pursuant to the title of the lemma, it suffices to rescale $\psi$ to a function $\tilde{\psi}$ whose principal curvatures are bounded away from zero, whose $C^3$ norm is $O(1)$, and whose domain has $O(1)$ dimensions. For then, Theorem \ref{t0.1} applies to the graph of $\Psi$, and subsequently scaling back to the original coordinates produces the caps of Theorem \ref{t0.2}.

In fact, once we know that the $C^3$ norm of $\tilde{\psi}$ is $O(1)$, we only need to verify that the Hessian of $\tilde{\psi}$ is approximately 1. This is because we will have determined inequalities $$|\l_1 + \l_2| \lesssim 1$$ $$|\l_1 \l_2| \gtrsim 1$$ holding true throughout $\tQ$, guaranteeing that the principal curvatures $\l_1$ and $\l_2$ are both $\sim 1$. 

Let us consider a judicious choice for $\tp$. Set $C = (A_{2,0}^{1/2}B_1B_2)^{-1}$, and define $\tilde{\psi}(\xi_1, \xi_2) = C\psi(B_1\xi_1, B_2\xi_2)$. Then, $$ \tilde{\psi}(\xi_1, \xi_2) = C(A_{2,0}B_1^2\xi_1^2 + A_{0,2} B_2^2 \xi_2^2 + A_{3,0} B_1^3 \xi_1^3 + A_{2,1} B_1^2B_2 \xi_1^2\xi_2$$ \begin{equation} \label{2.2} + A_{1,2} B_1B_2^2\xi_1\xi_2^2 + A_{0,3}B_2^3\xi_2^3 + O(|(B_1\xi_1, B_2\xi_2)|^4)), \end{equation} and $\xi_1, \xi_2 \lesssim 1$.

Termwise differentiation applied to \eqref{2.2} suffices when verifying $C^3$ boundedness. Since the coordinate lengths are $O(1)$ by hypothesis, we are only required to check that the coefficients for each term are $O(1)$, and the specifications fixed upon $B_1$ and $B_2$ ensure this in particular for the coefficient of the $\xi_1^3$-term. As well, noting that by hypothesis $$K \sim A_{2,0},$$ we may check that $$\det D^2\tilde{\psi} = \tilde{\psi}_{\xi_1\xi_1}\tilde{\psi}_{\xi_2\xi_2} - (\tilde{\psi}_{\xi_1\xi_2})^2 = C^2B_1^2B_2^2\det D^2\psi \sim A_{2,0} A_{2,0}^{-1} = 1$$ in both of the cases mentioned above.

Referring back to the harmonic analysis, rescaling $\tQ$ amounts to rewriting $$f(x) = \int_{Q \times [-\delta, \delta]} \hat{f}(\xi, \psi(\xi) + v) e(x \cdot (\xi, \psi(\xi) + v)) d\xi dv$$ as $$g(B_1x_1, B_2x_2, C^{-1}x_3) = B_1B_2C^{-1} \int_{Q' \times [-C\delta, C\delta]} \hat{f}(B_1 \xi_1, B_2 \xi_2, \psi(B_1\xi_1, B_2\xi_2) + v) \cdot $$ $$e((B_1x_1, B_2x_2, C^{-1}x_3) \cdot (\xi_1, \xi_2, C\psi(B_1\xi_1, B_2\xi_2) + v)) d\xi dv,$$ using change of variables. ($Q'$ is a rectangle with $O(1)$ dimensions.) The functions $g$ and $f$ comprise our first example of affine equivalence. $g$ is Fourier supported in the $C\delta$-neighborhood of the graph of $\tilde{\psi}$, so Theorem \ref{t0.1} yields a decoupling of $g$ over caps $\t'$ of dimensions $(C\delta)^{1/2}$. In turn, reversing the previous change of variables yields that each $\|P_{\Nc_{C\delta}(\t')}g\|_p$ may be replaced with $\|P_{\Nc_\delta(\t)}f\|_p$ within inequality \eqref{*} or \eqref{**}, where $\t$ has dimensions $|A_{2,0}|^{-1/2}\delta^{1/2} \times \delta^{1/2}$. In cases where only an $\ell^p$ decoupling is possible, the final decoupling constant will be $$C_\epsilon \delta^{-\epsilon}((C\delta)^{-1})^{1/2-1/p} = C_\epsilon \delta^{-\epsilon} (A_{2,0}^{1/2}B_1B_2 \delta^{-1})^{1/2-1/p} = C_\epsilon \delta^{-\epsilon} |\{\t\}|^{1/2-1/p}.$$ 

Let us confirm that $\t$ is flat for each case. Notice that we may assume that \begin{equation} \label{**2.1} A_{3,0}^{-2}A_{2,0}^3 \geq \delta. \end{equation} For if not, then the caps described by the hypothesis are contained within flat caps, as follows. When $A_{3,0} \lesssim \sqrt{A_{2,0}}$, $$A_{2,0} \xi_1^2 + A_{0,2}\xi_2^2 \leq (1 + A_{0,2})A_{2,0}^2 \lesssim \delta,$$ with the third-order Taylor error being clearly $O(\delta)$. While if $\sqrt{A_{2,0}} \lesssim A_{3,0}$, $$A_{2,0} \xi_1^2 + A_{0,2} \xi_2^2 \leq (1 + A_{0,2})A_{2,0}^3A_{3,0}^{-2} \lesssim \delta$$  and the third-order Taylor error is confirmed to be $O(\delta)$ too by the last inequality. It can be confirmed in both cases that the dimension bounds specified at the end of the lemma hold in this scenario.

It is already clear of course that the second order terms in \eqref{2.2} are $O(\delta)$ when $q \in \t$, but \eqref{**2.1} guarantees that the third-order Taylor error for this particular principal expansion is $O(\delta)$ too. It is a direct check that we will leave to the reader.

\end{proof}

In the sequel, it will be enough to secure length bounds $$B_1 \leq \sqrt{|A_{2,0}|}, \qquad B_2 \leq |A_{2,0}|$$ or alternatively $$B_1 \leq \frac{|A_{2,0}|}{|A_{3,0}|}, \qquad B_2 \leq \frac{|A_{2,0}|^{3/2}}{|A_{3,0}|}.$$ For then, a type II decoupling scheme, described in the next section, with respect to the parabolic cylinder $$(\xi_1, \xi_2) \mapsto A_{0,2} \xi_2^2$$ secures \eqref{1.7}.

An application of Lemma \ref{l2.1} is the following. We shall need it for proving Proposition \ref{pr2.0}, which establishes the decoupling theory for caps significantly larger than those addressed in the previous lemma.

\begin{lem} \label{l2.2}
 Let $\delta > 0$. For each $\delta^{1/3} \leq 2^{-k} \leq 1$, partition $\{(\xi_1, \xi_2): 2^{-k} \leq |\xi_1| \leq 2^{-k+1}, |\xi_2| \leq 1 \}$ into axis-parallel rectangles $\Rc_k$ of dimensions essentially $2^{k/2}\delta^{1/2} \times \delta^{1/2}$. As well, let $\Rc$ denote axis-parallel rectangles of dimensions $\delta^{1/3} \times \delta^{1/2}$ that partition the remaining portion of the set where $0 \leq |\xi_1| \leq \delta^{1/3}$. We will distinguish rectangles lying on the right of the $\xi_2$ axis and those lying on the left as $\Rc_k^+$ and $\Rc_k^-$ respectively.
 
 Then, $\{\tilde{\Rc_k}\} \cup \{\tilde{\Rc}\}$ is a decoupling partition comprised of flat caps of the graph of $$\Phi(\xi_1, \xi_2) = \xi_1^3 + \xi_2^2, \qquad (\xi_1, \xi_2) \in [-1,1]^2.$$  

The decoupling inequality appears as \begin{equation} \label{**2.3} \|f\|_p \leq C_\epsilon \delta^{-\epsilon} ((\sum_\Rc \|P_{\Nc_\delta(\tilde{\Rc})} f\|_p^2+ \sum_k \sum_{\Rc_k^+ }\|P_{\Nc_\delta(\tilde{\Rc}_k^+)}f\|_p^2)^{1/2}$$ $$+ \sum_k |\{\Rc_k^- \}|^{1/2-1/p} (\sum_{\Rc_k^-} \|P_{\Nc_\delta(\tilde{\Rc}_k^-)}f\|_p^p)^{1/p}),\end{equation} holding true for $f$ Fourier supported within the $\delta$ neighborhood of $\G_\Phi$.

\end{lem}

\begin{proof}

Let $\delta > 0$ be given. The Gaussian curvature is computed as $\sim \xi_1$. Therefore, we first apply the triangle inequality in order to partition dyadically the $(\xi_1, \xi_2)$-plane into axis-parallel slabs $\Sc, \Sc_k^\pm$ of horizontal width $2^{-k}$ and distance $2^{-k}$ from the line $\xi_1 = 0$, $\delta^{1/3} \leq 2^{-k} \lesssim 1$. Initially, we have $$\|f\|_p \leq \|P_{\Nc_\delta(\tilde{\Sc})}f\|_p + \sum_{2^{-k} > \delta^{1/3}} (\|P_{\Nc_\delta(\tilde{\Sc}_k^+)}f\|_p + \|P_{\Nc_\delta(\tilde{\Sc}_k^-)}f\|_p).$$ 

Within the region $\Sc = \{(\xi_1, \xi_2): |\xi_1| \leq \delta^{1/3}\}$, \begin{equation} \label{l1} \Phi(\xi_1, \xi_2) = \xi_2^2 + O(\delta), \end{equation} whereas within $|\xi_1| \sim 2^{-k},$ we have $$\Phi(\xi_1, \xi_2) = \xi_2^2 + O(2^{-3k}).$$ We may then apply cylindrical decoupling with respect to the parabolic cylinder $$(\xi_1, \xi_2) \mapsto \xi_2^2$$ in order to obtain $\ell^2$ decoupling partitions with caps projecting down as $$\{(\xi_1, \xi_2) : |\xi_1| \in [0, \delta^{1/3}), \xi_2 \in [b, b+\delta^{1/2})\} $$ and caps $\th_k^\pm$ projecting down as $$\{(\xi_1, \xi_2) : |\xi_1| \in [2^{-k}, 2^{-k+1}), \xi_2 \in [b, b+2^{-3k/2})\}$$ respectively. We then apply H\"older's inequality in order to derive an $\ell^p$ decoupling inequality over the caps $\th_k^- \subset \tilde{\Sc}_k^-$: \begin{equation} \label{*2.3} \|P_{\Nc_\delta(\tilde{\Sc}_k)}f\|_p \leq  C_\epsilon \delta^{-\epsilon} |\{\th_k^-\}|^{1/2-1/p}(\sum_{\th_k^-} \|P_{\Nc_\delta(\th_k^-)} f\|_p^p)^{1/p}. \end{equation}

The $\delta^{1/3} \times \delta^{1/2}$ caps are flat; however, the $2^{-k} \times 2^{-3k/2}$ caps $\th_k$ must be partitioned further. For this, we shall resort to Lemma \ref{l2.1}, once we have arranged the appropriate setting. Lemma \ref{l2.1} applies only to the principal expansions of $\G_\Phi$, a fact which almost impedes further decoupling of $\th_k$ via rescaling since it is unwieldy algebraically to determine the principal expansions explicitly. But, as far as the Fourier integral is concerned, such an expansion of the graphing function for each $\th_k$ can be easily obtained by a simple linear reparametrization of $\Phi$. Setting \begin{equation} \label{*2.4} \xi_1 = \xi_1' + 2^{-k}, \qquad \xi_2 = \xi_2' + b, \end{equation} $\Phi(\xi_1, \xi_2)$ becomes \begin{equation} \label{*2.5} \Phi(\xi_1', \xi_2') = \xi_2'^2 + \xi_1'^3 + 3(2^{-k})\xi_1'^2 + 2b\xi_2' + 3(2^{-2k})\xi_1' + 2^{-3k} + b^2. \end{equation} By change of variables then, $$f(x) = \int \hat{f}(\xi_1' + 2^{-k}, \xi_2' + b, \Phi(\xi_1', \xi_2') + v) e(x \cdot(\xi_1' + 2^{-k}, \xi_2' + b, \psi(\xi_1', \xi_2') +v)) d\xi' dv$$ and by simple algebra, we obtain $$|f(x)| = |\int \hat{f}(\xi_1' + 2^{-k}, \xi_2'+b, \Phi(\xi_1', \xi_2') + v) e((x_1 + 3(2^{-2k})x_3, x_2 +2bx_3, x_3) \cdot$$ \begin{equation} \label{*2.5} (\xi_1', \xi_2', \xi_2'^2 + 3(2^{-k})\xi_1'^2 + \xi_1'^3+v)) d\xi' dv|. \end{equation} The graphing function \begin{equation} \label{*2.6*} \tilde{\Phi}(\xi_1', \xi_2') = \xi_2'^2 + 3(2^{-k})\xi_1'^2 + \xi_1'^3 \end{equation} appearing in \eqref{*2.5} is indeed a principal expansion, and notably its coefficients and the dimension bounds on $\xi_1'$ and $\xi_2'$ satisfy the second case of the rescaling lemma. Applying Lemma \ref{l2.1} and undoing the change of variables expressed in \eqref{*2.4}, we thus obtain caps $\tilde{\Rc}_k$ of dimensions $2^{k/2}\delta^{1/2} \times \delta^{1/2}$ that partition each $\th_k$ with corresponding decoupling constant \begin{equation} \label{*2.4*} C_\epsilon \delta^{-\epsilon} |\{\tilde{\Rc}_k \subset \th_k^-\}|^{1/2-1/p} \end{equation} for those $\th_k^- \subset \tilde{\Sc}_k^-$.
The full $\ell^p$ decoupling constant thus obtained at this stage is simply the multiplication of that in inequality \eqref{*2.3} with \eqref{*2.4*}.

\end{proof}

\begin{re} Before continuing, let us comment on how rescaling allows us to extend the result of Lemma \ref{l2.2}. Namely, consider how we might decouple a function Fourier supported in the $\delta$-neighborhood of $$\tP(\xi_1, \xi_2) = \xi_2^2+ A\xi_1^3, \qquad (\xi_1, \xi_2) \in [-1, 1] \times [-1, 1]$$ for some arbitrary $O(1)$ $A$. Composing $\tP$ with the linear map \begin{equation} \label{**2.61} (\xi_1, \xi_2) \mapsto (A^{-1/3}\xi_1, \xi_2)\end{equation} results in $\Phi$ restricted to the domain $[-A^{1/3}, A^{1/3}]\times [-1,1]$. Then we apply Lemma \ref{l2.2} as written and undo the change of variables \eqref{**2.61} to acquire caps of dimensions at most $A^{-1/3}\delta^{1/3} \times \delta^{1/2}$ that comprise a decoupling partition of $\G_{\tP}$. \end{re}

The value of Lemma \ref{l2.2} is that we may now approximate local regions of more general $\M$ by the truncation $A_{0,2}\xi_2^2 + A_{3,0}\xi_1^3$ of \eqref{1.3}. In this way, we will be able to decouple more sizable caps within $\M$ into caps having the specified dimensions required for Lemma \ref{l2.1}.

\section{Decoupling for $\G_\Psi$} \label{s6}

As developed in \cite{BDK} and also just demonstrated for $\Phi$, the crucial step in obtaining decouplings for surfaces $\M$ with vanishing Gaussian curvature $K$ is to deal individually with the decoupling for each ``annulus" $\Ac_k \subset \M$ having essentially constant Gaussian curvature $ \pm 2^{-k}$. Consequently, in the next proposition, we restrict attention to a rectangular cap $\tQ$ of suitable dimensions that lies entirely within some $\Ac_k$. We will show that we may reduce the size of $\tQ$ via $\ell^2$ decoupling to dimensions small enough for rescaling, so long as $\M$ is just $C^4$.

It is noteworthy that Proposition \ref{pr2.0} addresses $\tQ$ whose $\xi_1$ principal directions are non-parallel with the tangent lines of $\Vc$. In other words, the lines of curvature are allowed to be transverse to the orientation of $\tQ$. As a reference point, in \cite{BDK}, the lines of curvature matched exactly with the level sets of $K$, which facilitated a much easier argument there.

\begin{pr} \label{pr2.0}

Let $\tQ \subset \Ac_k$ be a general principal expansion not containing planar points, and let $\tQ$ project onto a rectangle $Q$ whose width is $2^{-k}$ and whose length is $2^{-k/2}$. Alternatively, let $Q$ be a $\delta^{1/4} \times \delta^{1/2}$ rectangle over which the Gaussian curvature $K$ has magnitude less than $\delta^{1/2}$. 

If $K$ has positive sign throughout $\tQ$, we have the following $\ell^2(L^p)$ decoupling over flat caps $\t$: \begin{equation} \label{**2.8} \|P_{\Nc_\delta(\tQ)} f\|_p \lesssim_\epsilon \delta^{-\epsilon} (\sum_{\t \subset \tQ} \|P_{\Nc_\delta(\t)} f\|_p^2)^{1/2}. \end{equation} Otherwise, \eqref{**2.8} must be modified to an $\ell^p(L^p)$ decoupling: \begin{equation} \label{***2.8} \|P_{\Nc_\delta(\tQ)} f\|_p \lesssim_\epsilon |\{\t \subset \tQ\}|^{1/2-1/p + \epsilon} (\sum_{\t \subset \tQ} \|P_{\Nc_\delta(\t)} f\|_p^p)^{1/p}. \end{equation}

Moreover for each $\delta > 0$, when $2^{-k} \geq \delta^{1/2}$, each $\t$ is an intersection with $\tQ$ of a $2^{k/2} \delta^{1/2} \times \delta^{1/2}$ rectangular cap with sides parallel to the principal directions. When $|K|$ is less than $\delta^{1/2}$ however, the intersections are with $\delta^{1/4} \times \delta^{1/2}$ axis-parallel rectangular caps instead. 
\end{pr}

\begin{re} It is a straightforward application of the mean value theorem to see that a cap contained in $\Ac_k$, $2^{-k} \geq \delta^{1/2}$, whose points $(\xi_1, \xi_2)$ (in the tangent plane) satisfy bounds $$|\xi_1| \leq 2^{k/2}\delta^{1/2}, \qquad |\xi_2| \leq \delta^{1/2},$$ is indeed flat. Thus, in general, the intersections of $2^{k/2}\delta^{1/2} \times \delta^{1/2}$ boxes with $\Ac_k$ are the final caps that conclude a decoupling scheme for $\Ac_k$. 

In fact, for most values of $k$, $\Ac_k$ is essentially partitioned by $2^{k/2}\delta^{1/2} \times \delta^{1/2}$ boxes. For the mean value theorem assures that $K$ can change by $2^{-k}$ only if the change in the coordinates is at least $2^{-k}$. In turn, the diameter of the boxes satisfies $$2^{k/2}\delta^{1/2} \leq 2^{-k}$$ so long as $2^{-k} \geq \delta^{1/3}$. Thus, so long as $\Ac_k$ does not lie within the exceptional region where Gaussian curvature is less than $\delta^{1/3}$, we can be certain that $\Ac_k$ is partitioned typically.\end{re}

\begin{proof} 

The argument here occurs in three stages, the first two consisting of many iterations of $\ell^2$ decoupling. The goal is to reduce the dimensions of our caps to those mandated by the rescaling lemma (Lemma \ref{l2.1}), at which point either a final $\ell^2$ or $\ell^p$ decoupling is obtained depending on the sign of the Gaussian curvature. The rectangular caps comprising the final decoupling partition will have dimension bounds of $2^{k/2}\delta^{1/2} \times \delta^{1/2}$.

In the first stage, we only aim to decouple into caps $\th_{final, 1}$ whose $\xi_2$ length gets increasingly smaller, up to the terminal value $\sim 2^{-k}$.
We will pay somewhat loose attention to how the $\xi_1$ lengths also decrease during this process, only for the sake of notational convenience. If we only utilize continually that the $\xi_1$ lengths are bounded by $2^{-k/2}$, a suitable bound on the number of iterative steps required is still obtained.
In the second stage, however, if necessary, we aim to lower instead the $\xi_1$ lengths to at most $2^{-k}/A_{3,0}$ first; and then it easily follows that we may also reduce the $\xi_2$ size to at most $2^{-3k/2}/A_{3,0}$. This will be achieved again by repeating the iteration from the first stage for each cap. For the terminal caps $\theta_{final, 2}$ thus obtained, we finish the argument by applying Lemma \ref{l2.1} to each of them during the third stage, thereby finally decoupling into caps that are flat at scale $\delta$.

Finally, the following key insight makes it possible to achieve $\ell^2$ decoupling at each iterative step. Note that the diameter of $\tQ$ is $2^{-k/2}$. Therefore, the $A_{3,0}$ coefficient within the principal expansion \begin{equation} \label{**2.9} \psi(\xi_1, \xi_2) = A_{2,0} \xi_1^2 + A_{0,2} \xi_2^2 + A_{3,0} \xi_1^3 + A_{2,1} \xi_1^2 \xi_2 + A_{1,2} \xi_1 \xi_2^2 + A_{0,3} \xi_2^3 + O(|(\xi_1, \xi_2)|^4), \end{equation} either remains essentially constant or else is essentially less than $2^{-k/2}$ in magnitude for every basepoint $q \in \tQ$. This is true in light of the mean value theorem and the fact that $\tQ$ is at least $C^4$.

\bigskip

\textit{Stage I}. Let us begin the first stage of the argument. Let $\epsilon > 0$ be given. 
The idea of the proof is simply to decouple \eqref{**2.9} repeatedly by focusing on $$\tP(\xi_1, \xi_2) = A_{0,2}\xi_2^2 + A_{3,0}\xi_1^3,$$ or $$\tP(\xi_1, \xi_2) = A_{0,2}\xi_2^2$$ in the case of small $A_{3,0}$, as an approximating surface of $\tQ$ at small scales. Throughout, we aim to incur only loss in the decoupling constant of the form $C_\epsilon \delta^{-\epsilon}$.

We may and do assume $A_{0,2} = 1$ (via simple rescaling) in order to simplify the notation. In what follows, we seek the decoupling partition of $\Nc_\delta(\tQ)$. As a result, we should add $O(\delta)$ to \eqref{**2.9} in our analysis below, but we shall actually ignore this term throughout and use various powers of $2^{-k}$ instead for the neighborhood width. If at any point those powers turn out to be less than $\delta$, our argument will conclude prematurely with the original Bourgain-Demeter parabolic decoupling of $$\psi(\xi_1, \xi_2) = 2^{-k} \xi_1^2 + \xi_2^2 + O(\delta).$$ The caps thus obtained have the dimension bounds $2^{k/2}\delta^{1/2} \times \delta^{1/2}$ as prescribed. 

Lastly, let us clarify that no problem arises from the presence of additional bounded, positive factors in our neighborhood widths. This is due to the fact that the iteration will only occur $O(|\log \epsilon|)$ many times, as we shall see. Our neighborhood widths will be derived from the Taylor expansion \eqref{**2.9}, so it is important here as well that $\tQ$ is $C^4,$ in order to ensure satisfactory neighborhood widths below.\\

Let us commence. By hypothesis, each term in \eqref{**2.9} except for the second, third, and fourth terms is $O(2^{-2k})$. The fourth term is $O(2^{-3k/2})$. Thus, \eqref{**2.9} simplifies initially to \begin{equation} \label{*2.6}\psi(\xi_1, \xi_2) = \xi_2^2 + A_{3,0}\xi_1^3 + O(2^{-3k/2}),\end{equation} so we may decouple $\G_\psi$ using Lemma \ref{l2.2} for the $O(2^{-3k/2})$ neighborhood of $$\tP(\xi_1, \xi_2) = \xi_2^2+ A_{3,0} \xi_1^3.$$ Some care must be exercised in choosing the basepoint $q$ for \eqref{**2.9}, as we seek an $\ell^2$ decoupling of \eqref{*2.6}, and Lemma \ref{l2.2} only provides this over the region $\Ec$ where $A_{3,0} \xi_1 > 0$. A case analysis intervenes here, for which it is crucial that $\tQ$ is $C^4$. 

First, we may assume that $A_{3,0}$ satisfies \begin{equation} \label{**2.60} |A_{3,0}| \gtrsim 2^{-k/2}\end{equation} throughout $Q$, since otherwise $$|A_{3,0} \xi_1^3| \leq 2^{-2k} \qquad \text{for every $(\xi_1, \xi_2) \in Q$}$$ and \eqref{*2.6} then simplifies further to \begin{equation} \label{**2.6} \psi(\xi_1, \xi_2) = \xi_2^2 + O(2^{-3k/2}). \end{equation} In this scenario, we may use $$\tP(\xi_1, \xi_2) = \xi_2^2$$ instead as our reference surface throughout the iterative decoupling procedure explained below, running parabolic cylindrical decoupling in $\ell^2(L^p)$ at each step.  

Given \eqref{**2.60}, we next consider the case in which the principal directions do not make an $O(2^{-k/2})$ angle with the orientation of $Q$. This situation is favorable because then the dimensions of $Q$ prevent its orientation from changing significantly with respect to the principal directions at points within $Q$. Thus, as illustrated in Figure \ref{f1}, we are able to choose the basepoint of \eqref{**2.9} such that $Q$ lies entirely within a region $\Ec$. If instead the $\xi_2$ axis makes an $O(2^{-k/2})$ angle with the short side of $Q$ having length $2^{-k}$, we may cover all of $Q$ by a region $\Ec$, except for a small triangle $\Tc$ having base length $2^{-3k/2}$ and height $2^{-k}$. No issue is posed here, as $\Tc$ has the dimensions required by Lemma \ref{l2.1} for rescaling, perhaps after two decouplings with respect to the cylinder $$(\xi_1, \xi_2) \mapsto \xi_2^2 + A_{0,3}\xi_2^3.$$ 

In all cases, our above decoupling is in $\ell^2(L^p)$. It yields axis-parallel rectangles $\Rc$ (not necessarily parallel to $Q$ but rather parallel to the principal directions) with $\xi_2$ dimension $2^{-3k/4}$ partitioning $Q$, perhaps after another additional basic decoupling. We label the caps in $\tQ$ lying above an $\Rc$ as $\th_1$. Let us note too in preparation for the next step in the iteration that the principal directions likely change for each $\th_1$. This detail was noted in \cite{LY}. However, the dimension bounds on $Q$ cause the size of the $\xi_2$ length to be preserved when taken according to the new principal direction. 

\begin{figure}
\begin{center}

\begin{tikzpicture}


\draw[gray, thick] (-5, 2) -- (-6, 4) -- (-2, 6) -- (-1, 4)-- cycle;

\node at (-1.7, 6) {$Q$};

\draw[gray, thick] (4, 2) -- (3, 4) -- (7, 6) -- (8, 4)-- cycle;

\node at (7.3, 6) {$Q$};

\draw[gray, thin, ->] (-3.5, 7.5) to [out = 60, in = 120] (5, 7.5);


\draw[black, thick, <->] (-5.5, 0.5) -- (-3.5, 6.5);
\draw[black, thick, <->] (-6.5,2.5) -- (-2, 1);

\draw[black, thick, <->] (2.5, 2) -- (3.5, 6);
\draw[black, thick, <->] (1, 4.5) -- (7, 3);


\draw[black, thick] (-5+ 1, 2 + 0.5) -- (-5 + 1, 2 + 3);
\draw[black, thick] (-5 + 2, 2 + 1) -- (-5+2, 2+3 + 0.5);
\draw[black, thick] (-5 + 3, 2 + 1.5) -- (-5+3, 2+3 + 1);

\draw[black, thick] (3.2, 3.6) -- (3.2, 4.1);
\draw[black, thick] (4, 2) -- (4, 2 + 2.5);
\draw[black, thick] (4+ 1, 2 + 0.5) -- (4 + 1, 2 + 3);
\draw[black, thick] (4 + 2, 2 + 1) -- (4+2, 2+3 + 0.5);
\draw[black, thick] (4 + 3, 2 + 1.5) -- (4+3, 2+3 + 1);


\draw[gray, thin] (-5.3, 2.6) to [out = 50, in = 130] (-4.8, 2.6);

\draw[gray, dashed] (3,4) -- (2, 6);
\draw[gray, thin] (2.7, 4.6) to [out = 50, in = 130] (3.2, 4.6);

\node at (-2.5,2) {$\Ec$};
\draw[gray, thick, ->] (-2.5, 2.2) -- (-3.4, 3.2);
\node at (5, 1.3) {$\Ec$};
\draw[gray, thick, ->] (5, 1.5) -- (4.7, 2.7);

\node at (-6.7, 5) {$\gtrsim 2^{-k/2}$};
\draw[gray, thick, ->] (-6.7, 4.8) -- (-5, 2.8);

\node at (2.6, 7) {$\gtrsim 2^{-k/2}$};
\draw[gray, thick, ->] (2.6, 6.8) -- (2.9, 4.8);

\node at (-6, 0) {principal axes};
\node at (2, 1.5) {principal axes};

\end{tikzpicture}

\caption{\textbf{Choosing a different basepoint over $Q \subset T_q(\M)$ enables $\tQ \subset \Ec$.}}
\label{f1}
\end{center}
\end{figure}
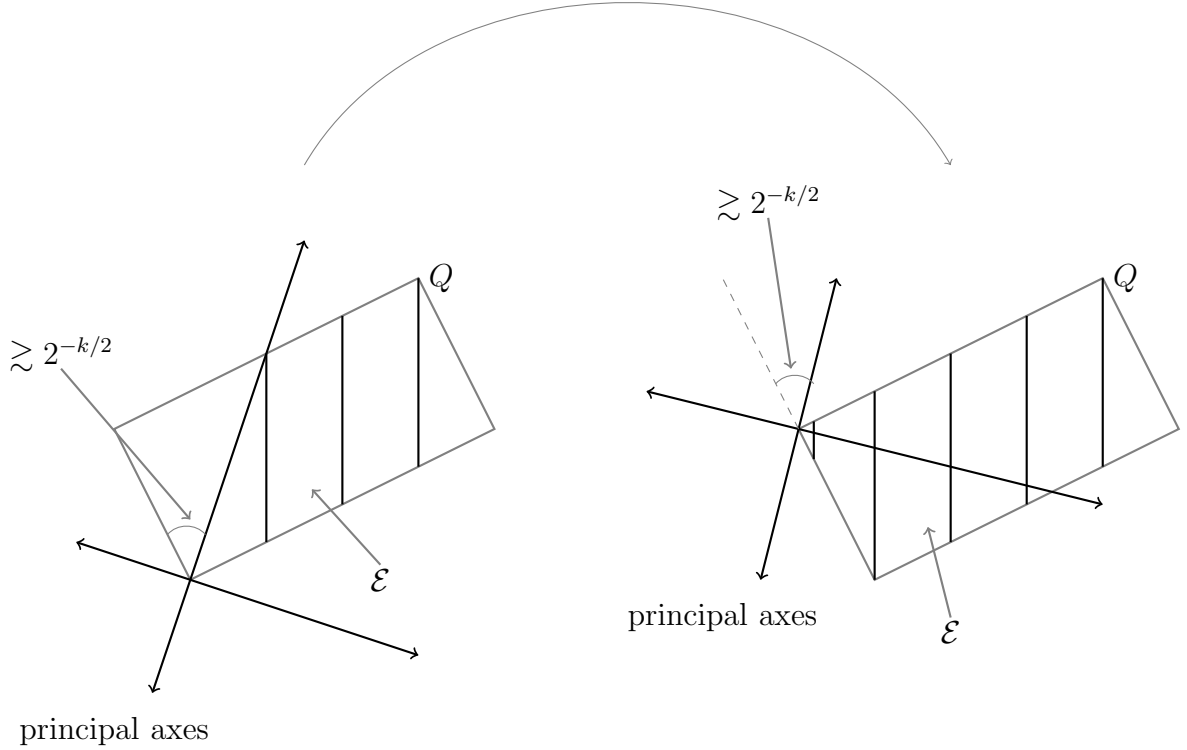

At the next step, we may keep $2^{-k/2}$ as the upper bound on the $\xi_1$ lengths, although this length may have been reduced. Then, for each $\th_1$, the term $A_{2,1} \xi_1^2\xi_2$ in \eqref{**2.9} is $O(2^{-7k/4})$, while all of the other terms in \eqref{**2.9} not used in the definition of $\tP$ are $O(2^{-2k})$. In what follows, the term $A_{2,1} \xi_1^2\xi_2$ will continue to provide the neighborhood width for our iterative decouplings. 

We have begun to launch an iterative decoupling scheme of the form introduced in Section \ref{s3}. The first iterative step gave us above the caps $\th_1 \subset \Nc_{\delta_1}(\G_{\tP})$ having their $\xi_2$ lengths bounded by $2^{-3k/4}$ with $$\delta_1 = O(2^{-7k/4}).$$ The corresponding decoupling inequality is $$\|f\|_p \leq C_\epsilon \delta_0^{-\epsilon} (\sum_{\Rc^{(1)} \cap Q \ne \emptyset} \|P_{\Nc_{\delta_1}(\th_1)}f\|_p^2)^{1/2},$$ where $\delta_0 = O(2^{-3k/2})$. At the second iterative step, for each $\th_1$, we have new $\xi_1$ and $\xi_2$ axes, corresponding to the principal directions at an appropriate point in $\th_1$, and we may again approximate the corresponding principal expansion $\psi$ for $\th_1$ by the reference surface $\G_{\tP}$. Using either Lemma \ref{l2.2} (with rescaling) or parabolic cylindrical decoupling if $A_{3,0}$ is small, we acquire caps $\th_2$ partitioning $\th_1$: \begin{equation} \label{*2.7}\|P_{\Nc_{\delta_1}(\th_1)}f\|_p \lesssim  C_\epsilon \delta_1^{-\epsilon}(\sum_{\Rc^{(2)} \cap \Rc^{(1)} \ne \emptyset} (\|P_{\Nc_{\delta_2}(\th_2)}f\|_p^2)^{1/2},\end{equation} where the $\th_2 \subset \tQ$ project onto rectangles $\Rc^{(2)}$ having their $\xi_2$ dimensions as $\delta_1^{1/2}$ and their $\xi_1$ dimensions varying as $2^{l/2}\delta_1^{1/2}$ with $\delta_1^{1/3} \leq 2^{-l} \leq 2^{-k/2}.$ Note that in general $\Rc^{(2)}$ is non-parallel with $\Rc^{(1)}$. Yet, again because of the diameter bound on $\tQ$, we may re-tile each $\Rc^{(1)}$ by new rectangles $\Rc^{(2)}$ now oriented parallel to $\Rc^{(1)}$, still maintaining \eqref{*2.7} . This change is permitted by Fourier projection since there are $O(1)$ overlaps between the old and new $\Rc^{(2)}$, in light of the diameter bound. It comes at the expense of an additional harmless constant factor of 25 in \eqref{*2.7} and is depicted in Figure \ref{f2}. 

For each $\th_2$, we take axes parallel to the principal directions at a suitable point in $\th_2$. The $\xi_2$ length of $\th_2$ (taken according to the new principal direction) satisfies $$ |\xi_2| \leq \delta_1^{1/2},$$ again guaranteed by the fact that $\th_1$ projects within a $2^{-k/2} \times 2^{-k/2}$ cube, and so it holds that $$\th_2 \subset \Nc_{\delta_2}(\G_{\tP})$$ 

\begin{figure}
\begin{center}

\begin{tikzpicture}


\draw[black, thick, ->]
(-10,2) -- (-10, 4);

\draw[black, thick, ->]
(-10,2) -- (-8,2);

\node at (-8.5, 1) {Principal axes at $q$};


\draw[gray, thick] (-8,4) -- (-10, 7) -- (-4, 11)--(-2, 8)--cycle;

\draw[black, thin, ->] (-2.5, 10) to [out = 40, in = 140] (2.5,10);

\draw[gray, thick] (1,4) -- (-1, 7) -- (5, 11)--(7, 8)--cycle;


\fill (-10,7) circle[radius=2pt];
\node at (-10.2, 7.2) {$q$};

\foreach \t in {0.05, 0.35, 0.65}
{
\draw[gray, thin] ({-8 + \t*6}, {4 + \t*4}) -- ({-8+\t*6}, {4+\t*4+4.3});
}

\draw[gray, thin] ({-10}, {7}) -- ({-10 +6.5}, {7});
\draw[gray, thin] ({-9}, {7.67}) -- ({-2.5}, {7.67});

\foreach \t in {0.05, 0.35, 0.65}
{
\draw[gray, thin] ({1 + \t*6}, {4 + \t*4}) -- ({1+\t*6}, {4+\t*4+4.3});
}

\draw[gray, thin] ({-1}, {7}) -- ({-1 +6.5}, {7});
\draw[gray, thin] ({0}, {7.67}) -- ({6.5}, {7.67});


\foreach \t in {-0.75, 0, 0.75, 1.5, 2.25}
{
\draw[black, thick] ({-7.5+\t}, {6.8 +\t*0.1}) -- ({-7.5-0.1 + \t}, {6.8 + 1 + \t*0.1});
}

\foreach \t in {0, 0.5, 1}
{
\draw[black, thick] ({-8.25 - \t*0.1}, {6.8 -0.75*0.1+\t}) -- ({-5.25 - \t*0.1}, {6.8 + 2.25*0.1 + \t});
}

\draw[black, dashed] ({-5.25 - 0.5*0.1 + 0.1}, {6.8 + 2.25*0.1 +0.5}) -- ({-5.25 - 0.5*0.1 + 0.1}, {10});
 
 \draw[black, dashed] ({-7.5-0.1 + 2.25}, {6.8 +1+2.25*0.1}) -- ({-7.5-3*0.1 + 2.25}, {6.8 + 1*3 + 2.25*0.1});
 
 \draw[gray, thin] ({-7.5 - 2*0.1 + 2.25}, {6.8 + 2 + 2.25*0.1}) to [out = 45, in = 135] ({-7.5 - 2*0.1 + 2.25 +0.25}, {6.8 + 2 + 2.25*0.1});
 
 \draw[gray, thin] ({-0.55}, {7.3}) -- ({5.9}, {7.3});
 \foreach \t in {0.25, 1, 1.75}
 {
 \draw[black, thick] ({-9.7 + \t + 10.9}, {7}) -- ({-9.7 + \t + 10.9}, {7.67});
 }
 
 \node at (-7, 10) {$O(2^{-k/2})$};
 \draw[black, thin, ->] (-7, 9.8) -- (-5.28, 8.6);
 
 \node at (-3, 5.5) {A typical $\Rc^{(i-1)}$};
 \draw[black, thin, ->] (-4, 5.7) -- (-4.7, 7.2);
 
 \node at (-4, 4) {Original $\Rc^{(i)}$};
 \draw[black, thin, ->] (-5, 4.2) -- (-6.5, 7.2);
 
 \node at (6,5.5) {New $\Rc^{(i)}$};
 \draw[black, thin, ->] (5.5, 5.7) -- (2.5, 7.1);
 
\end{tikzpicture}

\caption{\textbf{The reorienting of the $\Rc^{(i)} \subset T_q(\M)$ via Fourier projection.}}
\label{f2}
\end{center}
\end{figure}
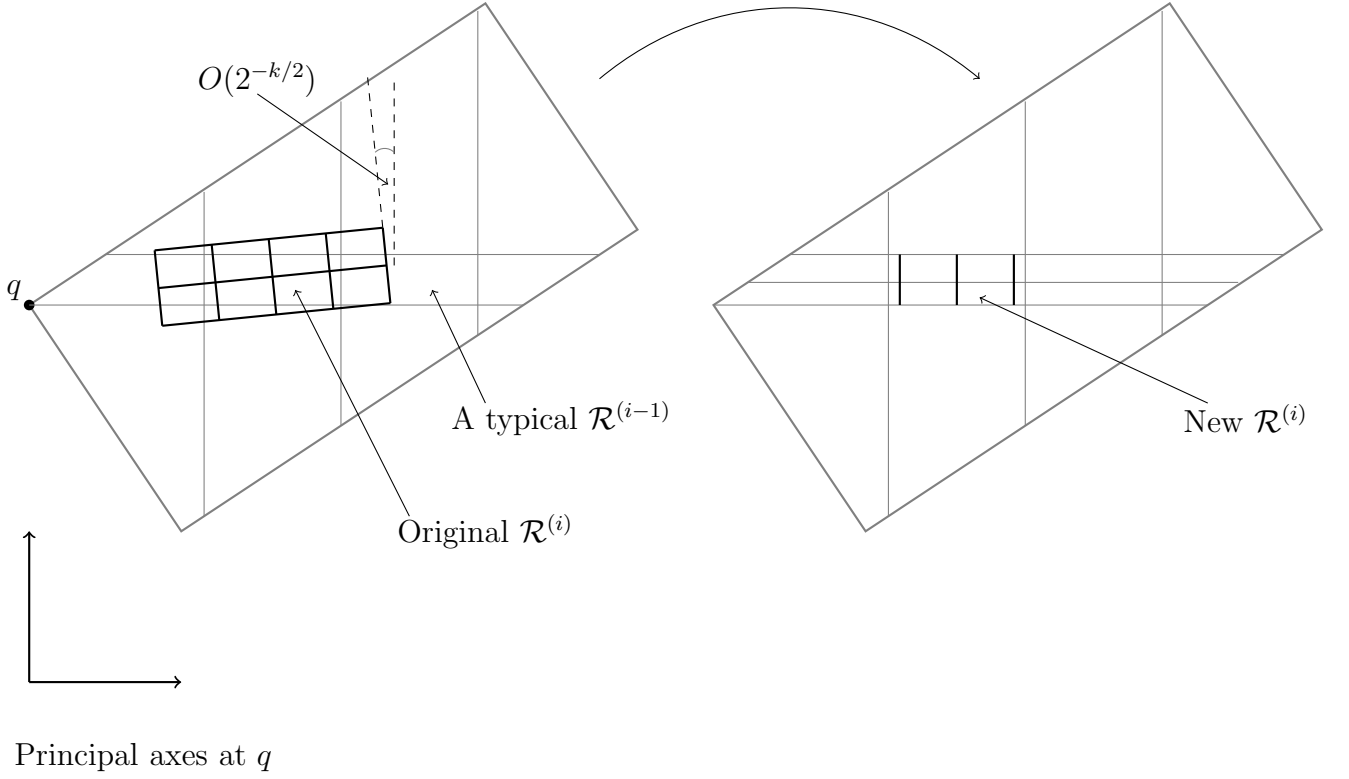

\ni with $$\delta_2 = O(2^{-k}\delta_1^{1/2})$$ by previous exposition. As the next iterative step, we partition $\Nc_{\delta_2}(\th_2)$ over caps $\th_3 \subset \Nc_{\delta_3}(\G_{\tP})$ projecting onto axis-parallel rectangles $\Rc^{(3)}$: $$\|P_{\Nc_{\delta_2}(\th_2)}f\|_p \lesssim  C_\epsilon \delta_2^{-\epsilon}(\sum_{\Rc^{(3)} \cap \th_2 \ne \emptyset} \|P_{\Nc_{\delta_3}(\th_3)}f\|_p^2)^{1/2} $$ where $$\delta_3 = O(2^{-k}\delta_2^{1/2}).$$ Each $\Rc^{(3)}$ has dimension bounds $$|\xi_1| \leq 2^{-k/2}, \qquad |\xi_2| \leq \delta_2^{1/2},$$ and may be modified via Fourier projection so as to run parallel to the initial rectangles $\Rc^{(1)}$. 

We continue this process until we obtain $$|\xi_2| \leq 2^{-k}$$ for each terminal cap. Let us confirm that the number of iterations required for that accomplishment is $O(|\log \epsilon|)$ and that the final decoupling constant has the form $C_\epsilon'\delta^{-O(\epsilon)}$. By induction, we can determine the neighborhood width $\delta_i$ for each iterative step $i$. Following the road map just illustrated, given a cap $\th_{i-1} \subset \Nc_{\delta_{i-1}}(\G_{\tP})$ having dimensions $$|\xi_1| \leq 2^{-k/2}, \qquad |\xi_2| \leq \delta_{i-2}^{1/2},$$ step $i$ will partition it into caps $\th_i \subset  \Nc_{\delta_i}(\G_{\tP})$ of dimensions $$|\xi_1| \leq 2^{-k/2}, \qquad |\xi_2| \leq \delta_{i-1}^{1/2}.$$ Since $A_{2,1}\xi_1^2 \xi_2$ is the largest among the error terms in \eqref{**2.9}, we take $\delta_i = 2^{-k}\delta_{i-1}^{1/2}.$ This procedure occurs for all $i$. Therefore, recalling $\delta_1 = 2^{-7k/4}$, it follows by induction that \begin{equation} \label{2.6**} \delta_i = 2^{-k\sum_{j=0}^{i+1} 2^{-j}}. \end{equation} Thus, the $\xi_2$ lengths of the caps $\th_i$ are each \begin{equation} \label{2.8} \delta_{i-1}^{1/2} = 2^{-k\sum_{j=1}^{i+2} 2^{-j}}.\end{equation}  As $i \rightarrow \infty$, the exponent in \eqref{2.8} goes to $-k$. Specifically, let $N \in \Nb$ be the first such that $$\sum_{j=1}^N 2^{-j} > 1 - \epsilon. $$ Then, $N = O(|\log \epsilon|)$, and at the $(N-2)$-th iteration, we will have caps with $\xi_2$ lengths at most $2^{-k}2^{k\epsilon}$. Basic decoupling then reduces this $\xi_2$ length further to the desired value $2^{-k}$ at the expense of the allowable constant $2^{k\epsilon/2} \leq \delta^{-\epsilon/2} $. 

Concerning the final decoupling constant, we know that each step $i$ introduces the additional decoupling constant $$C C_\epsilon(\delta_{i-1})^{-\epsilon} \leq C C_\epsilon 2^{2k\epsilon} \leq C C_\epsilon \delta^{-2\epsilon}.$$ Of course, these factors are accumulating and thus the full decoupling constant attached to a $\|P_{\th_i}f\|_p$ at step $i$ may be taken as \begin{equation} \label{2.7} (CC_\epsilon)^i (\delta^{-2\epsilon})^i. \end{equation} Since only $O(|\log \epsilon|)$ iterative steps occur at this stage in the proof, their product terminates here as $$(CC_\epsilon)^{O(|\log \epsilon|)}\delta^{-O(\epsilon |\log \epsilon|)},$$ a permissible constant for decoupling (since $\epsilon \log \epsilon \rightarrow 0$ as $\epsilon \rightarrow 0$). 

We have obtained final bounds on the lengths of caps $\th_{final, 1}$: \begin{equation} \label{2.8*} |\xi_1| \leq 2^{-k/2}, \qquad |\xi_2| \leq 2^{-k}, \end{equation} partitioning $\tQ$. Since the width of $Q$ is $2^{-k}$, we know that each $\th_{final, 1}$ is approximated by an axis-parallel rectangle of the same dimensions as $\th_{final, 1}$, a detail which ensures that we get the expected $\ell^p$ decoupling constant in Stage III.

\bigskip

\textit{Stage II}. This stage is only necessary in the event that $A_{3,0} \sim 2^{-l}$ such that $$2^{-l} \gtrsim 2^{-k/2}.$$ Otherwise, we move directly to Stage III.

Assuming then that $A_{3,0}$ is large, we now shift focus to reducing the size of the $\xi_1$ component to the bound expressed in the second case of Lemma \ref{l2.1}. Step I certainly brought progress because the local graph parametrizations \eqref{**2.9} now appear as \begin{equation} \label{2.9} \psi(\xi_1, \xi_2) = \xi_2^2 + A_{3,0} \xi_1^3 + O(2^{-k})\xi_1^2 + O(2^{-2k})\xi_1 + O(2^{-3k}).\end{equation}  As a reminder, our target dimensions are \begin{equation} \label{2.10} |\xi_1| \leq \frac{2^{-k}}{A_{3,0}}, \qquad |\xi_2| \leq \frac{2^{-3k/2}}{A_{3,0}}.\end{equation} If it ever holds for a cap $\th$ that $|\xi_1| \leq 2^{-k}$, then the $\xi_1$ length of $\th$ is sufficiently small, since $2^{-k} \lesssim \frac{2^{-k}}{A_{3,0}}.$ And once the $\xi_1$ length meets inequality \eqref{2.10}, just one more decoupling will give us the required bound on $|\xi_2|$. 

Thus, \eqref{2.9} simplifies to \begin{equation} \label{2.11} \psi(\xi_1, \xi_2) = \xi_2^2 + A_{3,0} \xi_1^3 + O(2^{-k})\xi_1^2 + O(2^{-3k}).\end{equation} We are ready to launch another iterative decoupling scheme. In this scenario, we are solely considered with reducing the size of the $\xi_1$ component, and we shall not keep track of how the $\xi_2$ length decreases. All that is needed is to ensure that at most $O(|\log \epsilon|)$ additional iterative decouplings need to be applied. At each iteration, we will again be decoupling with respect to $\G_{\tP}$ as the reference surface. As before, the dimensions of each cap $\th$ obtained will be preserved when we transfer to coordinates respecting the principal basis at $\th$, in light of the dimension bounds \eqref{2.8*}. The $O(2^{-k})\xi_1^2$ term in \eqref{2.11} will be contributing the neighborhood widths, in light of $$2^{-k}\xi_1^2 \leq O(2^{-3k})\qquad \Rightarrow \qquad |\xi_1| \lesssim 2^{-k}.$$

When executing step $i$ of the iterative decoupling procedure, the principal expansions appear as \begin{equation} \label{II}\psi(\xi_1, \xi_2) = \xi_2^2 + A_{3,0} \xi_1^3 + O(2^{-k}(\frac{2^{-k}}{A_{3,0}})^{ 2(1/3)\sum_{j=0}^{i-2} (2/3)^j}) \end{equation} with the dimensions being

 $$ |\xi_1| \leq (\frac{2^{-k}}{A_{3,0}})^{(1/3)\sum_{j=0}^{i-2}(2/3)^j}, \qquad |\xi_2| \leq 2^{-k}.$$Since $$\frac{1}{3} \sum_{j \geq 0} (\frac{2}{3})^j = 1,$$ the $\xi_1$ length goes to $2^{-k}/A_{3,0}$ as hoped. By an argument entirely similar to the one above, we see that again only $O(|\log \epsilon|)$ many iterative steps ultimately need to be executed. Thus, we conclude with $| \xi_1| \leq \frac{2^{-k}}{A_{3,0}}$ throughout each of the final caps at the expense of an allowable decoupling constant. 

We are now essentially done, because every term in \eqref{**2.9} except $\xi_2^2$ is $O(\frac{2^{-3k}}{A_{3,0}^2}),$ including the fourth-order Taylor error. Decoupling once with respect to the graph of $$(\xi_1, \xi_2) \mapsto \xi_2^2$$ gives us an $\ell^2$ decoupling of $\tQ$ into caps $\th_{final, 2}$. 

\bigskip

\textit{Stage III}. The above arguments provided three ways that we might arrive at caps with dimensions meeting the criteria of Lemma \ref{l2.1}. If $K < 0$ over $\tQ$, Lemma \ref{l2.1} can only provide an $\ell^p$ decoupling, so we must first utilize H\"older's inequality in order to lift our current $\ell^2$ decoupling to an $\ell^p$ decoupling. Then, we  apply that lemma to each cap to complete the proof.

\end{proof}

Proposition \ref{pr2.0} should be viewed as a sizable step towards our goal. For it resolves the decoupling theory of any surface for which the Gaussian curvature vanishes to first order as a function of distance from its vanishing set when this set is curved. An example is provided in the first case of the following proof. 

\begin{pr} \label{pr2.1} 
Let $A, B, C,$ and $D$ be $O(1)$ real constants with $|C| \sim 1$. For each $\delta > 0$, there exists a decoupling partition $\Pc_\delta$ with maximally flat, rectangular elements $\t$ of the graph $\G_\Psi$ of \begin{equation} \label{*2.12} \Psi(\xi_1, \xi_2) = A\xi_1^2 + \xi_2^2 + B\xi_1^3+ C\xi_1^2\xi_2 + D\xi_1^4, \qquad |\xi_1|, |\xi_2| \leq 1.\end{equation} Each $\t$ has the dimensions prescribed in Proposition \ref{pr2.0}. Moreover, the decoupling obtained here is a combined $\ell^2(L^p)$ and $\ell^p(L^p)$ inequality in the form of \eqref{***}.
\end{pr}

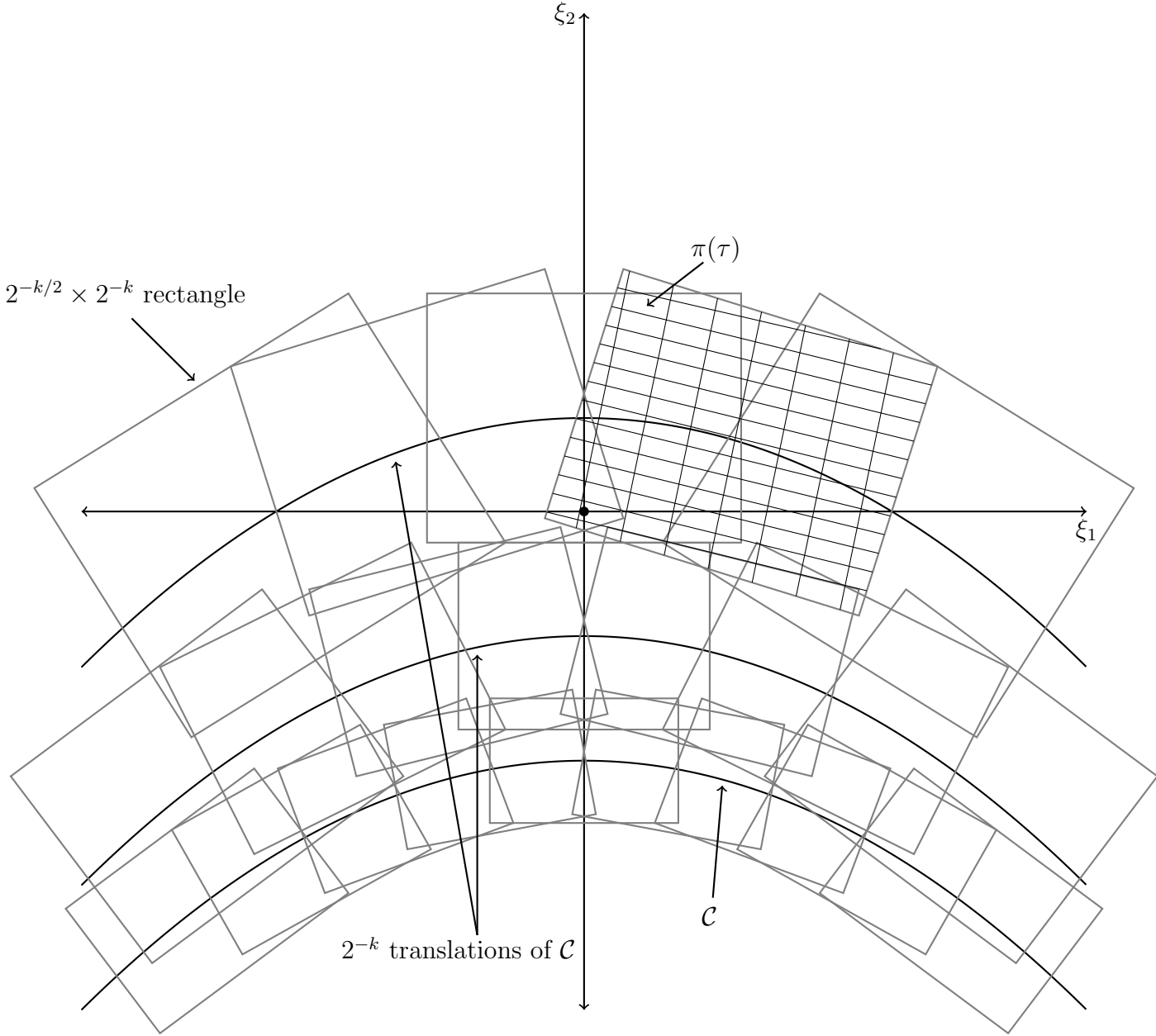
\begin{figure}
\begin{center}

\begin{tikzpicture}


\draw[black, thick, ->]
(0,0) -- (8,0);

\node at (8, -0.3) {$\xi_1$};

\draw[black, thick, <-]
(-8,0) -- (0,0);

\draw[black, thick, <->]
(0, -8) -- (0,8); 

\node at (-0.3, 8) {$\xi_2$};

\draw (0,0) circle[radius=2pt];
\fill (0,0) circle[radius=2pt];
\draw[black, thick]
({-8}, {-64/16 - 4})
\foreach \t in {-8, -7.99, ..., 8}
{
--({\t}, {-\t*\t/16 - 4})
};

\draw[black, thick]
({-8}, {-64/16 - 4+2})
\foreach \t in {-8, -7.99, ..., 8}
{
--({\t}, {-\t*\t/16 - 4+2})
};

\draw[black, thick]
({-8}, {-2.5})
\foreach \t in {-8, -7.99, ..., 8}
{
--({\t}, {-\t*\t/16 +1.5})
};

\foreach \t in {-6, -4.5, -3, -1.5, 0, 1.5, 3, 4.5, 6}
{
\draw[gray, thick] (\t-1.5-\t/8, -\t*\t/16-4+ 1.5*\t/8-1) -- (\t-1.5+\t/8, -\t*\t/16-4 + 1.5*\t/8 + 1) -- (\t+1.5+\t/8, -\t*\t/16- 4 - 1.5*\t/8 + 1) -- (\t+3/2-\t/8, -\t*\t/16 - 4-1.5*\t/8 - 1) -- cycle;
}

\foreach \t in {-6, -4, -2, 0, 2, 4, 6}
{
\draw[gray, thick] (\t-2-1.5*\t/8, -\t*\t/16-4+ 2*\t/8-1.5+2) -- (\t-2+1.5*\t/8, -\t*\t/16-4 + 2*\t/8 + 1.5 +2) -- (\t+2+1.5*\t/8, -\t*\t/16- 4 - 2*\t/8 + 1.5 +2) -- (\t+2-1.5*\t/8, -\t*\t/16 - 4-2*\t/8 - 1.5 +2) -- cycle;
}

\foreach \t in {-5, -2.5, 0, 2.5, 5}
{
\draw[gray, thick] (\t-2.5-2*\t/8, -\t*\t/16-4+ 2.5*\t/8-2+5.5) -- (\t-2.5+2*\t/8, -\t*\t/16-4 + 2.5*\t/8 + 2 +5.5) -- (\t+2.5+2*\t/8, -\t*\t/16- 4 - 2.5*\t/8 + 2 +5.5) -- (\t+2.5-2*\t/8, -\t*\t/16 - 4-2.5*\t/8 - 2 +5.5) -- cycle;
}

\foreach \t in {-2, -1.3, ..., 2.5}
{
\draw[black, thin] ({2.5 + \t -2*2.5/8}, {-2.5*2.5/16-4 - 2.5*\t/8-2+5.5}) -- ({2.5+(\t-0.4) +2*2.5/8}, {-2.5*2.5/16-4 - 2.5*(\t-0.4)/8 + 2 +5.5});
}

\foreach \t in {-1.9, -1.6, ..., 1.6}
{
\draw[black, thin] ({2.5 -2.5 +\t*2.5/8}, {-2.5*2.5/16-4 + 2.5*2.5/8+\t+5.5}) -- ({2.5+2.5 +(\t+0.3)*2.5/8}, {-2.5*2.5/16-4 - 2.5*2.5/8 + (\t +0.3) +5.5});
}

\draw[black, thin] ({2.5 -2.5 +1.7*2.5/8}, {-2.5*2.5/16-4 + 2.5*2.5/8+1.7+5.5}) -- ({2.5+(2-0.4) +2*2.5/8}, {-2.5*2.5/16-4 - 2.5*(2-0.4)/8 + 2 +5.5});


\node at (2.1, 4.2) {$\pi(\t)$};
\draw[thick, ->] (1.9, 4) -- (1, 3.3);

\node at (2, -6.5) {$\Cc$};
\draw[thick, ->] (2.05, -6.2) -- (2.2, -4.4);

\node at (-2, -7) {$2^{-k}$ translations of $\Cc$};
\draw[thick, ->] (-1.7, -6.8) -- (-1.7, -2.3);
\draw[thick, ->] (-1.7, -6.8) -- (-3, 0.8);

\node at (-7.3, 3.5) {$2^{-k/2} \times 2^{-k}$ rectangle};
\draw[thick, ->] (-7.2, 3.1) -- (-6.2, 2.1);

\end{tikzpicture}

\caption{\textbf{The orthogonal projection $\pi$ onto the $(\xi_1, \xi_2)$ plane of the caps $\t \in \Pc_\delta(\G_\Psi)$ contained within one $2^{-k/2} \times 2^{-k}$ rectangular cap inside of $\Ac_k$.}}
\end{center}
\end{figure}

\begin{proof}

Let us commence the proof. We know that $\G_\Psi$ is non-planar since its Hessian matrix never vanishes. Planarity of a graph $\G_\Xi$ means that the Hessian matrix of $\Xi$ vanishes at a point, and the chain rule would then imply that the second-order derivatives of any function composition involving $\Xi$ vanish too. Thus, quick computation shows that the Hessian matrices of all principal expansions of $\G_\Psi$ are nonzero. By \eqref{*1}, the Gaussian curvature is computed to be

\begin{equation} \label{2.9*} K \sim 2(2A+6B\xi_1 + 2C\xi_2 + 12D\xi_1^2) - (2C\xi_1)^2,\end{equation} demonstrating in particular that $K$ vanishes on a connected set $\Cc$. Note that the principal curvatures of $\G_\Psi$ are continuous functions of $\xi$. We conclude that only one principal curvature $\l_1$ vanishes at points in $\Cc$ and therefore is essentially the Gaussian curvature throughout the domain of $\Psi$.

According to \eqref{2.9*}, the level sets of $K$ at dyadic powers $\pm 2^{-k}$ are vertical translations by $\pm \frac{2^{-k}}{C}$ of the parabola $\Cc$ \begin{equation} \label{2.9**} \xi_2 = C^{-1}((C^2 - 6D)\xi_1^2  - 3B\xi_1-A).\end{equation} 
Note that $|K| \sim 2^{-k}$ within the $2^{-k}$-neighborhoods $\Nc_k$ of the above translations of $\Cc$. As well, if $C > 0$, the region above $\Cc$ is where $K > 0$ and where $\ell^2$ decoupling will be confirmed. Otherwise, it is the region below $\Cc$.

We proceed by dividing the matter into two cases. Let $G = C^2 - 6D$. \\

\textit{Case I}\\

If $|G| \geq (1/13)C^2$, then $\ell^2$ decoupling with respect to the parabola $\Cc$ holds for $\Nc_k$, and it may be extended to the vertical cylinder in $\R^3$ above it. The caps thus given have $\xi_1$ lengths $2^{-k/2}$, and this implies that their $\xi_2$ length is $O(2^{-k/2})$ too, in consideration of \eqref{2.9**}. Proposition \ref{pr2.0} thus obtains the flat decoupling partition for each $\G_{\Psi|_{\Nc_k}}$. The full decoupling inequality for $\|f\|_p$ is then acquired via the triangle inequality, as at the beginning of the proof of Lemma \ref{l2.2}. The flat caps thus obtained have diameter at most $\delta^{1/4}$.\\

\textit{Case II}\\

Now we address the situation where \begin{equation} \label{*2.10} |G| < (1/13)C^2,\end{equation} which only occurs if $D > 0$. 

In the current context, the smallness of $G$ prevents an immediate minute dissection of $\Ac_k$. Nevertheless, decoupling still enables us to partition $\Nc_k$ as almost rectangular pieces of width $2^{-k}$ for all values of $G$, not just $G = 0$ (in which case the $\Nc_k$ are exactly rectangles). This is because Theorem \ref{t0.1} dissects $\Nc_k$ as $2^{-k}$ vertical neighborhoods of caps projecting onto the $\xi_1$ axis as intervals of length $G^{-1/2}2^{-k/2}$. These caps are approximated by rectangles arranged along the tangent directions for the parabola $\Cc$.

Thus, via this $\ell^2$ decoupling if $G \ne 0$, we are able to restrict attention to a rectangle $\Bc$ of width $2^{-k}$ that lies essentially within $\Nc_k$. Alternatively, it will be more convenient in the exposition below to handle $\Bc$ as an honest $2^{-k}$ vertical neighborhood of a cap: $$ \Bc_+ = \{(a+ \xi_1, C^{-1}(G(a+\xi_1)^2 - 3B(a+\xi_1) - A) + v): \xi_1 \in [0, G^{-1/2}2^{-k/2}], $$ \begin{equation} \label{*2.101} v \in [C^{-1}2^{-k}, C^{-1}2^{-k+1}] \} \end{equation} or $$ \Bc_- = \{(a+ \xi_1, C^{-1}(G(a+ \xi_1)^2 - 3B(a + \xi_1) - A) - v): \xi_1 \in [0, G^{-1/2}2^{-k/2}], $$ \begin{equation} \label{*2.102} v \in [C^{-1}2^{-k}, C^{-1}2^{-k+1}] \}. \end{equation}Note that \eqref{*2.10} implies that $D \sim 1$. This insight will play a key role in our iterative decoupling scheme below. For each $\Bc = \Bc_{\pm}$, we have the following iterative argument for reducing the length of $\Bc$ to the desired bound of $2^{-k/2}$ via iterative $\ell^2$ decoupling.

To set the stage, we make two more helpful observations. First, we note that the form of $\Psi$ is almost preserved by affine equivalence when taking its Taylor expansion at any point $(a, b) \in \R^2$: $$\Psi(\xi_1, \xi_2) = A(\xi_1 - a + a)^2+ (\xi_2-b + b)^2 +  B(\xi_1 - a + a)^3 + C(\xi_1 - a + a)^2(\xi_2-b +b) + D(\xi_1 - a + a)^4 $$  $$ = A'(\xi_1 - a)^2 + (\xi_2 - b)^2 + B'(\xi_1-a)^3 + C(\xi_1  - a)^2(\xi_2 - b) + D(\xi_1-a)^4 +$$ \begin{equation} \label{*2.11}E(\xi_1 - a)(\xi_2 - b) + F_1(\xi_1 - a) + F_2(\xi_2 - b) + F_3 \end{equation} where \begin{equation} \label{*2.111} A' = A + 3aB + bC + 6a^2D, \end{equation} \begin{equation} \label{*2.112} B' = B + 4aD, \end{equation} and \begin{equation} \label{*2.113} E = 2aC. \end{equation} Affine equivalence allows us to simplify \eqref{*2.11} to \begin{equation} \label{*2.12} \Psi(\xi_1', \xi_2') = A' \xi_1'^2+ \xi_2'^2 + B'\xi_1'^3 + C\xi_1'^2\xi_2' + D\xi_1'^4 + E\xi_1'\xi_2', \end{equation} where $\xi_1' = \xi_1 - a$ and $\xi_2' = \xi_2 - b$.

Our other insight occurs by inserting the Gaussian curvature into \eqref{*2.12}. In terms of the $\xi_1', \xi_2'$ coordinates, the Gaussian curvature $K'$ of $\G_\Psi$ satisfies \begin{equation} \label{*2.13} K' \sim 2(2A'+6B'\xi_1' + 2C\xi_2' + 12D\xi_1'^2) - (2C\xi_1' + E)^2. \end{equation} Recall that $K|_{\G_{\Psi|_{\Nc_k}}} \sim 2^{-k}$. Then we also have that $K' \sim 2^{-k}$ since the Hessian of $\Psi$ is preserved under the change in coordinates $(\xi_1, \xi_2) \rightarrow (\xi_1', \xi_2').$ Solving \eqref{*2.13} for $\xi_2'$ thus yields \begin{equation} \label{*2.14} \xi_2' = C^{-1}(O(2^{-k}) + (C^2-6D)\xi_1'^2 + (CE- 3B')\xi_1' - A' + \frac{E^2}{4}). \end{equation} Now, we seek to substitute \eqref{*2.14} into \eqref{*2.12}. For this purpose, we borrow some of the $C\xi_1'^2 \xi_2'$ term in \eqref{*2.12}, for example:  $$\Psi(\xi_1', \xi_2') = \xi_2'^2 + A'\xi_1'^2 + B'\xi_1'^3 + \frac{9}{10}C\xi_1'^2\xi_2' + D\xi_1'^4 + E\xi_1'\xi_2' + \frac{1}{10}C\xi_1'^2\xi_2'$$ yielding $$ \Psi(\xi_1', \xi_2') = \xi_2'^2 + A'\xi_1'^2 + B'\xi_1'^3 + \frac{9}{10}C\xi_1'^2 \xi_2'+ D\xi_1'^4 +E\xi_1'\xi_2'+$$ $$ \frac{1}{10}C\xi_1'^2(\frac{1}{C}(O(2^{-k})+ (C^2-6D)\xi_1'^2 + (CE- 3B')\xi_1'  - A' + \frac{E^2}{4})),$$ which simplifies to \begin{equation} \label{*2.15} \Psi(\xi_1', \xi_2') = \xi_2'^2 + A''\xi_1'^2 + B''\xi_1'^3 + \frac{9}{10}C\xi_1'^2 \xi_2'+ (\frac{1}{10}C^2 + \frac{2}{5}D)\xi_1'^4 + E\xi_1'\xi_2' + O(2^{-k})\xi_1'^2, \end{equation} where \begin{equation} \label{*2.151} A'' = \frac{1}{10}(\frac{E^2}{4} + 9A'), \end{equation} and \begin{equation} \label{*2.152} B'' = \frac{1}{10}(CE + 7B'). \end{equation}

The error in \eqref{*2.15} serves well as neighborhood widths for a type II decoupling scheme that achieves our end goal, so long as \begin{equation} \label{*2.153} \tilde{\Psi}(\xi_1', \xi_2') = \xi_2'^2 + A''\xi_1'^2 + B''\xi_1'^3 + \frac{9}{10}C\xi_1'^2 \xi_2'+ (\frac{1}{10}C^2 + \frac{2}{5}D)\xi_1'^4 + E\xi_1'\xi_2' \end{equation} fits case I. The reader may be concerned by the presence of the $E\xi_1'\xi_2'$ term in \eqref{*2.153}. However, a quick check shows that the only impact of this term is to cause vertical translation of the parabola $\tilde{\Cc}$ (the set where the Gaussian curvature of the graph of $\tPs$ vanishes). In particular, it does not affect the curvature of $\tilde{\Cc}$.

And $\tilde{\Cc}$ does indeed have nonzero curvature, as we now deduce. In light of \eqref{*2.10}, \begin{equation} \label{*2.16} D > (2/13) C^2.\end{equation} In turn, we compute that the Gaussian curvature of the graph of $\tPs$ vanishes at points $(\xi_1', \xi_2')$ satisfying \begin{equation} \label{6.3.2.1} \xi_2' = \frac{5}{18C}(\frac{3}{5}(\frac{7}{5}C^2 - 16D)\xi_1'^2 + (\frac{18}{5}CE - 12B'')\xi_1' + E^2 - 4A'', \end{equation} which is a curved parabola $\tilde{\Cc}$ by \eqref{*2.16}. 

Now, we specialize $b$ to when $(a,b)$ is the vertex of $\Bc_\pm$ expressed as $$(a, C^{-1}(Ga^2 - 3Ba - A \pm 2^{-k})).$$ This simplifies \eqref{6.3.2.1} to \begin{equation} \label{6.3.2.2} \xi_2' = \frac{5}{18C}(\frac{3}{5}(\frac{7}{5}C^2 - 16D)\xi_1'^2 + \frac{6}{5}(4aC^2 - 7B - 28D)\xi_1' \mp \frac{18}{5}2^{-k}). \end{equation} Our goal is to ensure that $\Bc_\pm$ lies within the region over which we know that the graph of $\tilde{\Psi}$ has an $\ell^2$ decoupling. Thus, we must ensure that $\Bc_\pm$ lies above the parabola $\tilde{\Cc}$. 

Without loss of generality, we may assume that $C > 0$, and then it holds that $C^{-1}(\frac{7}{5}C^2 - 16D) < 0$. Thus, $\tilde{\Cc}$ opens downward, a convenient detail that motivates why an $\ell^2$ decoupling should be possible. Observe that if we take $\Bc = \Bc_+$, so that the minus sign is taken in \eqref{6.3.2.2}, then $\Bc$ evidently lies above $\tilde{\Cc}$ by that same expression. (Recall that $(a,b)$ is equal to the origin in $\xi_1', \xi_2'$ coordinates.) Consequently, the scenario is reduced to the first case considered above, where the argument there yields an $\ell^2$ decoupling of $\delta$-neighborhoods of $\G_{\tilde{\Psi}}$ over caps of dimensions at most $\delta^{1/4} \times \delta^{1/2}$ for arbitrary $\delta > 0$. We may thus run an iterative decoupling scheme, at each iterative step using $\tilde{\Psi}$ as an approximating surface for a cap $\th \subset \G_\Psi$, fixing the supremum over $\th$ of $O(2^{-k})\xi_1'^2$ from \eqref{*2.15} to be our error $\delta$, and applying the $\ell^2$ decoupling. At step $i$, the $\xi_1'$ length will be reduced to $$|\xi_1'| \leq (2^{-k})^{(1/4)\sum_{j=0}^{i-1} (1/2)^j}.$$ This bound converges to $2^{-k/2}$ as $i \rightarrow \infty$, so we ultimately achieve $$|\xi_1'| \leq 2^{-k/2}, |\xi_2'| \leq 2^{-k}.$$ And this accomplishment comes at the expense of an allowable decoupling constant, in light of the constant given by the $\ell^2$ decoupling at each iterative step and the number of steps required. Proposition \ref{pr2.0} then completes the decoupling of $\|P_{\Nc_\delta(\tilde{\Bc}_+)} f\|_p$.

Finally, we address $\Bc = \Bc_-$. Now it holds that the plus sign in \eqref{6.3.2.2} must be taken, and we further assume that the original parabola $\Cc$, associated to $\Psi$, is downward pointing, the case of $\Cc$ opening upward following more easily from the next paragraph. We investigate briefly how we might arrange an $\ell^2$ decoupling using some $\tPs$.

For $\tPs$ based directly on $\Bc$, $\tilde{\Cc}$ is in fact above a portion of $\Bc$. However, we may consider whether shifting horizontally our choice of $\tPs$ supplies a $\tilde{\Cc}$ that does extend entirely below $\Bc$. The type II decoupling scheme just described never reduces $\xi_1$ lengths to a value smaller than $2^{-k/2}$. Thus, we examine whether $\tilde{\Cc}$ decreases faster by a difference of $2^{-k}$ than does $\Cc$ over $\xi_1$ intervals of length $d \geq 2^{-k/2}$. There is good reason to believe such a statement as $\tilde{\Cc}$ is significantly more curved than $\Cc$. Working in $\xi_1', \xi_2'$ coordinates and quoting \eqref{6.3.2.2}, the relevant desired inequality appears as $$ C^{-1}[(C^2 - 6D)(\ta \pm d)^2 - 3B(\ta \pm d) - (C^2 - 6D)\ta^2 + 3B\ta] \geq \frac{5}{18C}[\frac{3}{5}(\frac{7}{5}C^2 - 16D)d^2$$ $$ + \frac{6}{5}(4\ta C^2 - 7B - 28\ta D)(\pm d)] + C^{-1}2^{-k}, $$ which is indeed true if \begin{equation} \label{2.19} \pm ( \ta + M_{B,C,D}) \geq |(\frac{23}{30}C^2 - \frac{10}{3}D)d - d^{-1}2^{-k}||-\frac{2}{3}C^2 + \frac{8}{3}D|^{-1} \end{equation} for some fixed $M_{B,C,D} \in \R.$ The value $-(2/3)C^2 + 8/3D$ is bounded away from zero by a constant dependent upon $C$. Thus, \eqref{2.19} informs us that within our iterative decoupling scheme, we may shift $a$ by an $O(1)$ enlargement of the length of $\Bc$ and then decouple in $\ell^2(L^p)$. Applying basic decoupling afterwards in each step, we thereby obtain the smaller length desired there with acceptable $O(1)$ loss. As well, at each iterative step, the sign choice required in \eqref{2.19} for its validity dictates whether we shift our choice of $\tPs$ to the left or right of $\Bc$. Figure \ref{f4} illustrates the above argument.

One last word concerning the aesthetics of the new type II iterative decoupling scheme. Concerning the shapes of the decoupling caps obtained throughout, these essentially are always parallel sub-rectangles of increasingly smaller length and width equal to that of $\Bc$. This is due to the fact that the above iterative decouplings involving $\tPs$ never decrease any length to a value smaller than $2^{-k}$, the width of $\Bc$. Therefore, it is only the length of $\Bc$ that is being reduced at each step, and we are permitted to replace each new cap obtained with a true subrectangle of $\Bc$ parallel to it by standard Fourier projection theory.

\end{proof}

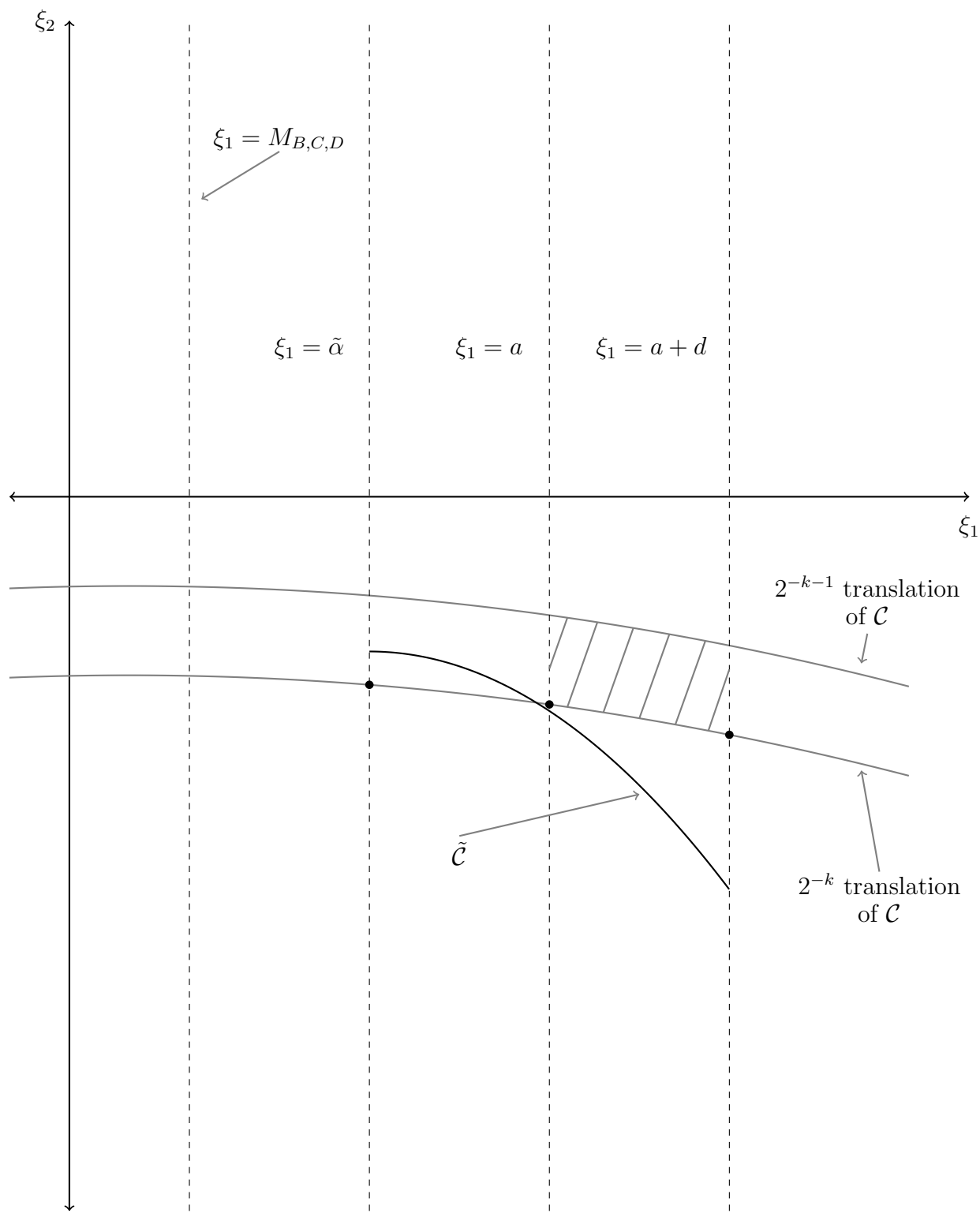
\begin{figure}
\begin{center}

\begin{tikzpicture}


\draw[black, thick, <->]
(-8,2) -- (8, 2);

\node at (8, 1.5) {$\xi_1$};
\draw[black, thick, <->]
(-7,-10) -- (-7,10);

\node at (-7.4, 10) {$\xi_2$};


\draw[gray, thick]
(-8, {-1/25 - 1})
\foreach \t in {-8, -7.99, ..., 7}
{
--({\t}, {-(\t+6)*(\t+6)/100 - 1})
};

\draw[gray, thick]
(-8, {-1/25+ 0.5})
\foreach \t in {-8, -7.99, ..., 7}
{
--({\t}, {-(\t+6)*(\t+6)/100 + 0.5 })
};
\fill (-2,{-16/100 - 1}) circle[radius=2pt];
\draw[black, dashed] (-2, -10) -- (-2, 10);
\node at (-3, 4.5) {$\xi_1 = \ta$};

\fill (1, {-49/100 - 1}) circle[radius=2pt];
\draw[black, dashed] (1, -10) -- (1, 10);
\node at (0, 4.5) {$\xi_1 = a$};

 \fill (4, {-1 - 1}) circle[radius=2pt];
\draw[black, dashed] (4, -10) -- (4, 10);
\node at (2.7, 4.5) {$\xi_1 = a + d$};

\draw[gray, thick] (1.3, {-7.3*7.3/100 + 0.5}) -- (1, {-49/100 - 0.4});
\draw[gray, thick] (1.8, {-7.8*7.8/100 + 0.5}) -- (1.3, {-7.3*7.3/100 - 1});
\draw[gray, thick] (2.4, {-8.4*8.4/100 + 0.5}) -- (1.9, {-7.9*7.9/100 - 1});
\draw[gray, thick] (3, {-9*9/100 + 0.5}) -- (2.5, {-8.5*8.5/100 - 1});
\draw[gray, thick] (3.6, {-9.6*9.6/100 + 0.5}) -- (3.1, {-9.1*9.1/100 - 1});
\draw[gray, thick] (4, {-1 + 0.1}) -- (3.65, {-9.65*9.65/100 - 1});

\draw[black, thick]
(-2, {- 0.6})
\foreach \t in {-2, -1.99, ..., 4}
{
--({\t}, {-(\t +2)*(\t +2)/9 - 0.6})
};

\draw[black, dashed] (-5, -10) -- (-5, 10);
\node at (-3.5, 8) {$\xi_1 = M_{B,C,D}$};
\draw[gray, thick, ->] (-3.5, 7.8) -- (-4.8, 7);

\node at (6.5, -4.5) {$2^{-k}$ translation};
\node at (6.5, -5) {of $\Cc$};
\draw[gray, thick, ->] (6.5, -4.3) -- (6.2, -2.6);

\node at (6.3, 0.5) {$2^{-k-1}$ translation};
\node at (6.3, 0) {of $\Cc$};
\draw[gray, thick, ->] (6.3, -0.3) -- (6.2, -0.8);

\node at (-0.5, -4) {$\tilde{\Cc}$};
\draw[gray, thick, ->] (-0.5, -3.7) -- (2.5, -3);

\end{tikzpicture}

\caption{\textbf{Shifting the basepoint to the left by $d$ in deriving $\tPs$ ensures that $\tilde{\Cc}$ lies below the shaded region.}}
\label{f4}
\end{center}
\end{figure}

\section{Proof of Theorem \ref{t0.2}} \label{s7}

Equipped with the decoupling theory of our $\delta$-approximating surface $\G_\Psi$ for $U= U_i$, we are now ready to prove Theorem \ref{t0.2}. Following the previous page, let us define $\Dec(U, \delta)$ to be the smallest positive constant $C$ satisfying the following: there exists a decoupling partition $\Pc_{\delta}$ of $U$ comprised of flat caps $\t$, each fully contained in some $\Ac_k$, such that \eqref{***} holds instead with constant $C$ for all $f$ Fourier supported in $\Nc_\delta(U)$. The width of each $\t$ is of course at most $\delta^{1/2}$. It is evident too that the length is bounded by $\delta^{1/4}$ since $$2^{k/2}\delta^{1/2} \leq \delta^{1/4}$$ for all $2^{-k} \geq \delta^{1/2}$, and $$A_{2,1} \sim 1$$ implies that some choice of $p = (\pm \delta^{1/4}, \pm \delta^{1/2})$ satisfies $$\psi(p) \sim_{A_{2,1}} \delta$$ if $A_{3,0} \lesssim_{A_{2,1}} \delta^{1/4}.$ The alternative case of $A_{3,0}$ being larger than $\delta^{1/4}$ implies that flatness is reached at scale $$|\xi_1| \leq |A_{3,0}|^{-1/3}\delta^{1/3} \leq \delta^{1/4}.$$

The proof now goes through as before. The inequality \begin{equation} \label{3.3} \|f\|_p \leq \Dec(U, \delta^{4/5}) ((\sum_{\th} \|P_{\Nc_\delta(\th)} f\|_p^2)^{1/2} + \sum_k |\{\th_k\}|^{1/2-1/p}(\sum_{\th_k} \|P_{\Nc_\delta(\th_k)}f\|^p)^{1/p})  \end{equation} is readily given to us. Now, concerning each $\th$, it has dimensions at most $\delta^{1/5} \times \delta^{2/5}$ with respect to a principal basis taken within $\th$. Therefore, $\th$ is the graph of the function \begin{equation} \label{3.4} \psi(\xi_1, \xi_2) = A_{2,0} \xi_1^2 + A_{0,2} \xi_2^2 + A_{3,0} \xi_1^3 + A_{2,1} \xi_1^2 \xi_2 + A_{4,0} \xi_1^4 + O(\delta), \end{equation} where $A_{2,1} \sim 1$ by the hypothesis of the theorem. Proposition \ref{pr2.1} thus partitions $\th$ and $\th_k$ into flat caps $\t$ and $\t_k$ that have dimensions equal to those expressed in Proposition \ref{pr2.0}, in particular at the expense of a decoupling constant $C_\epsilon \delta^{-\epsilon}|\Pc_\delta(\th_k)|^{1/2-1/p}$ for the $\th_k$ terms in \eqref{3.3}. The constant $C_\epsilon$ necessarily depends on the $C^5$ norm of $\psi$ also. By the definition of $\Dec(U, \delta)$, we conclude the following inequality, which may be iterated: \begin{equation} \label{3.2} \Dec(U, \delta) \lesssim_{\epsilon, \M} \delta^{-\epsilon}\Dec(U, \delta^{4/5}). \end{equation} 

\section{Appendix}

The role of this appendix is to answer some questions relating to the geometry of $\M$ near $\Vc$. First, we shall derive a sufficiently explicit description of $\M$ near $\Vc$. Then, we shall see that $\M$ is not convex in this region and why this fact prohibits a full $\ell^2$ decoupling over maximally flat caps there.

\subsection{Characterization of $\M$ near $\Vc$} \label{a1}

This subsection uncovers some relevant ties linking the Euclidean description of $\M$ to that provided by its lines of curvature. Let us introduce notation that builds upon the exposition of Section \ref{s2}. Recalling the curvature parametrization, let $$C = \|\x_u(0,0)\|, \qquad C' = \|\x_v(0,0)\|,$$ (both nonzero necessarily), and define the projection map $\pi : \R^3 \rightarrow \R^2, \pi = (\xi_1, \xi_2),$ with respect to the principal directions spanned respectively by $\{\e_1, \e_2\}$ of $T_p(\M)$. The Jacobian of $\pi \circ \x(u,v)$ is nonsingular at $(0,0)$ and can be written in general as \begin{equation} \label{a.1} \begin{pmatrix} \frac{\partial \xi_1}{\partial u} & \frac{\partial \xi_1}{\partial v} \\ \\ \frac{\partial \xi_2}{\partial u} & \frac{\partial \xi_2}{\partial v} \end{pmatrix},\end{equation} where $\frac{\partial \xi_1}{\partial u} = \e_1 \cdot \x_u, \frac{\partial \xi_1}{\partial v} = \e_1 \cdot \x_v, \frac{\partial \xi_2}{\partial u} = \e_2 \cdot \x_u$, and $\frac{\partial \xi_2}{\partial v} = \e_2 \cdot \x_v$.

From the usual formula for the inverse of $2 \times 2$ matrices, we have the Jacobian of $(\pi \circ \x)^{-1}$ as \begin{equation} \label{a.2} \begin{pmatrix} \frac{\partial u}{\partial \xi_1} & \frac{\partial u}{\partial \xi_2} \\ \\ \frac{\partial v}{\partial \xi_1} & \frac{\partial v}{\partial \xi_2} \end{pmatrix} = \begin{pmatrix} \frac{\partial \xi_1}{\partial u} & \frac{\partial \xi_1}{\partial v} \\ \\ \frac{\partial \xi_2}{\partial u} & \frac{\partial \xi_2}{\partial v} \end{pmatrix}^{-1}  = J^{-1} \begin{pmatrix} \frac{\partial \xi_2}{\partial v} & -\frac{\partial \xi_1}{\partial v} \\ \\ -\frac{\partial \xi_2}{\partial u} & \frac{\partial \xi_1}{\partial u} \end{pmatrix}, \end{equation} where $J = (\frac{\partial \xi_1}{\partial u} \frac{\partial \xi_2}{\partial v} - \frac{\partial \xi_1}{\partial v}\frac{\partial \xi_2}{\partial u})$.

Evaluating \eqref{a.1} at $0 = (0,0)$, we get $$ \frac{\partial \xi_1}{\partial u}(0) = C, \quad \frac{\partial \xi_2}{\partial u}(0)= 0 = \frac{\partial \xi_1}{\partial v}(0), \quad \frac{\partial \xi_2}{\partial v}(0) = C', $$ which of course implies by \eqref{a.2} that \begin{equation} \label{a.3}\frac{\partial u}{\partial \xi_1}(0) = C^{-1}, \frac{\partial v}{\partial \xi_2}(0) = C'^{-1}, \frac{\partial u}{\partial \xi_2}(0) = 0 = \frac{\partial v}{\partial \xi_1}(0).\end{equation} As well, from \eqref{a.2}, we may compute \begin{equation} \label{a.4} \frac{\partial^2 v}{\partial \xi_1^2}(0) = -(J^{-1}\frac{\partial u}{\partial \xi_1}  \frac{\partial^2 \xi_2}{\partial u^2})(0) = -C^{-2}C' \frac{\partial^2 \xi_2}{\partial u^2}(0), \end{equation} which we shall use shortly. We are now ready to prove the following lemma.

\begin{lem*}  \label{A1} Assume that $\l_1(\bq) = 0$, $(\x(0,0) = \bq)$. Then, the curvature of $u \mapsto \x(u, 0)$ at $u = 0$ is $$-C'^{-1} \frac{\partial^2 v}{\partial \xi_1^2}(0),$$ and also \begin{equation} \label{a*} \frac{\partial^3 \psi}{\partial \xi_1^2 \partial \xi_2}(0) = C'\l_2(\bq) \frac{\partial^2 v}{\partial \xi_1^2}(0).\end{equation}  \end{lem*}

\begin{proof}
Let $\Cc$ denote the curve $u \mapsto \x(u, 0).$ First, we quote the following formulas for the principal curvatures, obtained from \cite{DoC}: \begin{equation} \label{a.5} \l_1 = \frac{\N \cdot \x_{uu}}{\x_u \cdot \x_u}, \qquad \l_2 = \frac{\N \cdot \x_{vv}}{\x_v \cdot \x_v}.\end{equation} Another formula from differential geometry (see \cite{DoC}) informs us that the curvature of $\Cc$ is given by the orthogonal projection of $\x_{uu}$ onto (\textbf{span}\{$\e_1\})^\bot \subset \R^3$ with scalar $\|\x_u \cdot \x_u\|^{-2}$.  Since $\l_1(\bq) = 0$, we know from \eqref{a.5} that $\x_{uu}$ lies in $T_\bq(\M)$. Thus, it follows that $C^{-2}\frac{\partial^2 \xi_2}{\partial u^2}(0) = C^{-2}\e_2 \cdot \x_{uu}(0)$ is the curvature at $0$ of the line of curvature $u \mapsto \x(u,0)$. The first assertion of the lemma is now taken care of by \eqref{a.4}.

Next, let us consider the derivatives of $\psi \circ (\pi \circ\x)$, which for ease of notation we shall also write as $\psi$. We would like to relate the derivatives of $\psi(\xi_1, \xi_2)$ to those of $\psi(u,v)$ as shown in \eqref{a*}. Repeated applications of the chain rule yield the following: \begin{equation} \label{a.7} \frac{\partial^2 \psi}{\partial \xi_1^2} = \frac{\partial^2 \psi}{\partial u^2}(\frac{\partial u}{\partial \xi_1})^2 + 2 \frac{\partial^2 \psi}{\partial u \partial v} \frac{\partial u}{\partial \xi_1}\frac{\partial v}{\partial \xi_1} + \frac{\partial \psi}{\partial u}\frac{\partial^2 u}{\partial \xi_1^2} + \frac{\partial \psi}{\partial v}\frac{\partial^2 v}{\partial \xi_1^2} + \frac{\partial^2 \psi}{\partial v^2}(\frac{\partial v}{\partial \xi_1})^2. \end{equation} Using the chain rule once more, we apply $\frac{\partial}{\partial \xi_2}$ to \eqref{a.7} and evaluate at $0$ to get 
\begin{equation} \label{a.8} \frac{\partial^3 \psi}{\partial \xi_1^2 \partial \xi_2}(0) = C'^{-1}(C^{-2}\frac{\partial^3 \psi}{\partial u^2 \partial v} + \frac{\partial^2 \psi}{\partial v^2} \frac{\partial^2 v}{\partial \xi_1^2})(0), \end{equation} making repeated use of the vanishing at $0$ expressed in \eqref{a.3}.

Consider the term $\frac{\partial^3 \psi}{\partial u^2 \partial v}(0)$. From pg$.\text{ }161$ of \cite{DoC}, we have the following identity for curvature parametrizations $\x$: \begin{equation} \label{a**} \N \cdot \x_{uv} \equiv 0. \end{equation} We differentiate \eqref{a**} once with respect to $u$ and obtain

\begin{equation} \label{a.9} \frac{\partial^2 \psi}{\partial u^2 \partial v}(0) = \N \cdot \x_{uuv}(0) = \N \cdot \x_{uuv}(0) + \N_u \cdot \x_{uv}(0) = \frac{\partial}{\partial u}(\N \cdot \x_{uv})(0) = 0,\end{equation} observing that $\N_u = -\l_1\x_u$ and $\l_1$ vanishes at $0$.

Therefore, \eqref{a*} is secured after we compute $\frac{\partial^2 \psi}{\partial v^2}(0)$. Observing that \begin{equation} \label{a.10} \psi(u,v) = \N(0) \cdot \x(u,v), \end{equation} we differentiate \eqref{a.10} twice. The result is $$\frac{\partial^2 \psi}{\partial v^2}(0) = (\N \cdot \x_{vv})(0) = C'^2 \l_2(\bq),$$ in light of \eqref{a.5}.

\end{proof}

\subsection{The non-convexity of $\M$} \label{a2}

The fact that \eqref{1.4} holds near the vanishing set of $K$ implies directly that $\M$ is not convex. For consider the Hessian of $\psi$ near $\bar{q}$ where $K(\bar{q}) = 0$. If we evaluate the Hessian of $\psi$ at points $(0, \xi_2)$, we get for $K = K(\xi_1, \xi_2)$ 

\begin{equation} \label{4.1} K \sim 2A_{2,1}\xi_2(2A_{0,2} + O(\xi_2)) - (O(\xi_2))^2 = 4A_{0,2}A_{2,1}\xi_2 + O(\xi_2^2) \end{equation} using the fact that $A_{2,0} = 0$ since $K(p) = 0.$ Since $A_{0,2} A_{2,1} \ne 0$, it follows that both $\M^+$ and $\M^-$ are nonempty. Convex hypersurfaces, on the other hand, are characterized by having nonnegative Gaussian curvature everywhere.

The fact that $K(q) < 0$ for some $q \in \M$ near $p$ implies that the $\delta$-neighborhoods of certain non-flat caps containing $q$ contain lines and hence have no $\ell^2$ decoupling partition. Let us verify this insight. Consider the principal expansion $\psi$, as expressed in \eqref{1.3}, based at such a point $q \in \M$. In this context, we have that $A_{2,0}$ and $A_{0,2}$ have opposite signs, and therefore the line $\mathcal{L}$ defined by $$\{(\xi_1, \sqrt{|\frac{A_{2,0}}{A_{0,2}}|}\xi_1, 0): |\xi_1| \leq \delta^{1/3}\}$$ is contained in $\Nc_\delta(U)$ for $\delta$ sufficiently small. Here, $U$ is the graph of $\psi$, and its size is taken so that the Gaussian curvature remains constant throughout $U$. 

Now, we know that the dimensions of flat caps within $\M$ having curvature $\sim 2^{-k}$ are at most $2^{k/2} \delta^{1/2} \times \delta^{1/2}$, so long as $\delta^{1/2} \leq 2^{-k}$. Therefore, letting $2^{-k}$ denote the approximate Gaussian curvature of $U$, we have, for all $\delta \leq 2^{-2k}$, that an $\ell^2$ decoupling into maximally $\delta$-flat caps would give a partition of $U$ into $2^{k/2}\delta^{1/2} \times \delta^{1/2}$ caps with constant $C_{\epsilon, \delta} = C_\epsilon \delta^{-\epsilon}$. But then, $\mathcal{L}$ would be partitioned at the expense of only $C_{\epsilon, \delta}$, in clear violation of Proposition \ref{p0.1} for $\delta > 0$ satisfying \begin{equation} \label{a1.1} 2^{k/2}\delta^{1/2} \leq \delta^{1/3 + \eta}\end{equation} for some fixed $0 < \eta $. 

The reader may observe that \eqref{a1.1} is only possible within the region $2^{-k} \geq \delta^{1/3}$. Thus, it seems to hint at the possibility that an $\ell^2$ decoupling might hold if we allow our caps to have dimensions $\delta^{1/3} \times \delta^{1/2}$, or $\delta^{1/3} \times \delta^{1/3}$ if we prefer square-like caps. More generally, the same may be conjectured for any hypersurface $\M$ with negative Gaussian curvature that does not contain lines. We seek to answer this conjecture in future work.


\begin{thebibliography}{99}
	

	\bibitem{BM} Bierstone, E. and Milman, P. D. {\em Semianalytic and subanalytic sets}, Publ. Math. I.H.E.S. \textbf{67} (1988), 5-42.
	\bibitem{BD3} Bourgain, J. and Demeter, C. {\em The proof of the $l^2$ decoupling conjecture}, Annals of Math. \textbf{182} (2015), no. 1, 351-389.

	\bibitem{BD4} Bourgain, J. and Demeter, C. {\em Decouplings for curves and hypersurfaces with nonzero Gaussian curvature}, J. Anal. Math. \textbf{133} (2017), 279-311.

	\bibitem{BDK} Bourgain, J., Demeter, C., and Kemp, D. {\em Decouplings for real analytic surfaces of revolution}, GAFA Israel seminar 2017-2019, Lecture Notes in Mathematics \textbf{2256}, 113-126.


	\bibitem{D} Demeter, C. {\em Fourier restriction, decoupling, and applications}, Cambridge Studies in Advanced Mathematics \textbf{184}. Cambridge University Press, Cambridge, 2020.

	\bibitem{DGW} Demeter, C., Guth, L., and Wang, H. {\em Small cap decouplings}, Geom. Funct. Anal. \textbf{30} (2020), 989-1062.

	

	\bibitem{DoC} do Carmo, Manfredo Perdigao. {\em Differential geometry of curves and surfaces}, Prentice-Hall Inc, 1976.
	
	\bibitem{GS} Garcia, R. and Sotomayor, J. {\em Lines of curvature on surfaces, historical comments, and recent developments}, S$\tilde{\text{a}}$o Paulo Journal of Mathematical Sciences \textbf{2} (2008), no.1, 99-143.

	\bibitem{K} Kemp, D. {\em Decouplings for surfaces of zero curvature}, arXiv preprint (2019), arXiv:1908.07002.
	

	\bibitem{KP} Krantz, S. and Parks, H. {\em A primer of real analytic functions}, Birkh\"auser Advanced Texts, Second Edition, Birkh\"auser Boston, New York 2002.
	
	\bibitem{LY} Li, J. and Yang, T. {\em Decoupling for mixed-homogeneous polynomials in $\R^3$}, arXiv preprint (2021), arXiv: 2104.00128.
	

	\bibitem{PS} Pramanik, M. and Seeger, A. {\em $L^p$ regularity of averages over curves and bounds for associated maximal operators}, Amer. J. Math. \textbf{129} (2007), no. 1, 61-103.




	

	
\end{thebibliography}
\end{document}